\documentclass{article}
\usepackage[utf8]{inputenc}
\usepackage{mathrsfs,amsmath}
\usepackage{secdot}
\sectiondot{subsection}
\usepackage[left=17mm,top=1in,bottom=1in]{geometry}
\usepackage{amsfonts,latexsym,amscd,amssymb}
\usepackage[english]{babel}
\usepackage{epsfig}
\usepackage{graphicx}
\usepackage{relsize,exscale}
\usepackage[titletoc,title]{appendix}
\usepackage{color}
\usepackage{amsthm}
\newtheorem{theorem}{Theorem}[section]

\newtheorem{lemma}[theorem]{Lemma}
\newtheorem{prop}[theorem]{Proposition}

\newtheorem*{remark}{Remarks}
\theoremstyle{definition}
\newtheorem{defn}[theorem]{Definition}
\numberwithin{equation}{section}
\title{Propagation of polarization sets for systems of MHD type}
\author{Rayhana Darwich}
\date{}

\begin{document}
	
	\maketitle
	\begin{abstract}
		Polarization sets were introduced by Dencker (1982) as a refinement of wavefront sets to the vector-valued case. He also clarified the propagation of polarization sets when the characteristic variety of the pseudodifferential system under study consists of two hypersurfaces intersecting tangentially (1992), or transversally (1995). In this paper, we consider the case of more than two intersecting characteristic hypersurfaces that are interesting transversally (and we give a note on the tangential case).
		
		Mainly, we consider two types of systems which we name "systems of generalized transverse type" and "systems of MHD type", and we show that we can get a result for the propagation of polarization set similar to Dencker's result for systems of transverse type. Furthermore, we give an application to the MHD equations.
	\end{abstract}
	\section{Introduction}\label{Introduction}
	In \cite{hormander1}, H\"{o}rmander defined the wavefront set of a distribution $u$, denoted by $\operatorname{WF}(u)$, which is a refinement of the singular support of a distribution. The wavefront sets does not only show the location of singularity, but also the direction in which the singularity occurs. Concerning the propagation of the wavefront sets, many results were given. For example, in \cite{hormander1} H\"{o}rmander gave the result for the propagation of the wavefront set for the solutions of partial differential equations, when considering the partial differential operator to be of real principal type, where he stated that the wavefront set is invariant under the bicharacteristic flow. Moreover, in \cite{denckerpropagationofsing}, Dencker studied the propagation of singularities for pseudodifferential operator $P$ on a smooth manifold $X$, having
	characteristics of variable multiplicity. He considered the characteristic set 
	to be union of hypersurfaces $S_j$, $j=1,...,r_0$ tangent at $\cap_{j=1}^{r_0} S_j$. Under some assumptions he proved that the wavefront set of the solution of the considered pseudodifferential operator, is invariant under the union of the Hamilton flows on $S_j$, $j=1,...,r_0$, given that $Pu$ is smooth on $X$. 
	
	In \cite{hormander1}, H\"{o}rmander defined locally the wavefront set of distributional sections $u\in \mathcal D'(X;E)$, where $E\rightarrow X$ is a vector bundle over the smooth manifold $X$. He defined the wavefront set of $u$ locally as $\bigcup \operatorname{WF}(u_j)$ where $(u_1,...,u_N)$ are the components of $u$ with respect to a local trivialization of $E$.  However, this definition does not specify in which components $u$ is singular, that is why Dencker defined in \cite{denckerprincipaltype} the polarization set for vector-valued distribution $u$ which we will denote by $\operatorname{Pol}(u)$. The polarization set also shows the location and the direction of the singularity as the wavefront set, but it additionally shows the most singular components of a distribution. Hence, the polarization set of a distribution is a refinement of the wavefront set, and the projection of $\operatorname{Pol}(u)\setminus 0$ on the cotangent bundle $T^*X$ gives the wavefront set of $u$. Similarly, the $H^s$-polarization set is defined as a refinement of the $H^s$-wavefront set, where $H^s$ denotes the usual Sobolev space.
	
	In \cite{denckerprincipaltype}, Dencker defined systems of pseudodifferential operators of real principal type; note that the definition of systems of pseudodifferential operators of real principal type differs from the case of scalar pseudodifferential operators of real principal type, and he defined Hamilton orbits for systems of real principal type which are certain line bundles, and then he proved that the polarization set of a solution $u$ of systems of real principal type $P$ will be union of Hamilton orbits, given that $Pu$ is smooth. In \cite{Gerardsobolev}, G\'{e}rard pointed out that the above result also holds for $H^s$-polarization sets. 
	
	Moreover, in \cite{denckerdoublerefraction}, Dencker considered pseudodifferential system having its characteristic set is union of two non-radial hypersurfaces intersecting tangentially at an involutive manifold of exactly order $k_0\geq 1$. He also assumed that the principal symbol vanishes of first order on the two-dimensional kernel at the intersection, and he assumed a Levi type of condition. Then, he defined systems satisfying these conditions to be systems of  uniaxial type. Outside the intersection of the hypersurfaces the system will be of real principal type, hence the propagation result of the polarization set is already known there. In this article, Dencker has also proved a propagation result of the polarization set at the intersection. In \cite{denckertransversaltype}, Dencker considered pseudodifferential system having its characteristic set is union of two non-radial hypersurfaces intersecting transversally at an involutive manifold of codimension $2$. He also assumed that the principal symbol vanishes of first order on the two-dimensional kernel at the intersection. Systems satisfying these conditions are systems of transverse type. In this article, Dencker has proved a propagation result of the polarization set at the intersection. Outside the intersection the system is of real principal type.
	
	We worked on extending Dencker's result stated  above to pseudodifferential systems having their characteristic sets is union of several non-radial hypersurfaces intersecting transversally at an involutive manifold of codimension $d_0\geq 2$; not necessary just two hypersurfaces as in the case of systems of transverse type and systems of uniaxial type. Note that even if we assumed that the hypersurfaces are intersecting tangentially of exactly order $k_0\geq 1$ instead of intersecting transversally, we get a similar result, and for the proof we use the same weight and metric introduced by Dencker in \cite{denckerdoublerefraction} for the symbol classes $S(\vartheta,g)$ of the Weyl calculus. We have considered two cases for that: the first case is the case where we have $r_0$ hypersurfaces, and we assumed that the $r_0$th-differential of the determinant of the principal symbol is different than zero at the intersection, and the $i$th-differential of the determinant of the principal symbol vanishes at the intersection for $i<r_0$. Moreover, we assumed that the dimension of the kernel of the principal symbol to be $r_0$ at the intersection, and we assumed a condition similar to the Levi type condition considered in \cite{denckerdoublerefraction}. We called systems satisfying the above conditions systems of generalized transverse type, and we proved that we have a similar propagation result of the polarization set as that for systems of transverse type. The second case, is the case where we also have $r_0$ hypersurfaces, and a condition similar to the Levi type condition considered in \cite{denckerdoublerefraction}, but here we assumed that the $(r_0+1)$th-differential of the determinant of the principal symbol is different than zero at the intersection, and the $i$th-differential of the determinant of the principal symbol vanishes at the intersection for $i<r_0+1$. Moreover, we assumed that the dimension of the kernel of the principal symbol to be $r_0+1$ at the intersection. We also assumed some additional conditions that we did not assume in the case of systems of generalized transverse type. We defined systems satisfying these conditions to be systems of MHD type. We named them systems of MHD type because we have first noticed such systems when we considered the linearized ideal MHD equations. Thus we will have a section in which we study the propagation of polarization sets for the linearized ideal MHD equations. 
	
	In our work, we will assume that we have $P\in\Psi^m_{phg}(X)$ an $N\times N$ system of pseudodifferential operators on a smooth manifold $X$ of order $m$. Let $p=\sigma(P)$ be the principal symbol, $\det p$ the determinant of $p$, and $\Sigma=(\det p)^{-1}(0)$ the characteristics of $P$. We consider $\Sigma$ to be union of several non-radial hypersurfaces intersecting transversally at an involutive manifold $\Sigma_2$.
	Now, we state our main theorem in this paper regarding the propagation of polarization sets for systems of generalized transverse type, and systems of MHD type, but its proof will be postponed to Sections $\ref{section Propagation of polarization sets for systems of generalized uniaxial type}$, and $\ref{section propagation of polarization sets for systems of MHD type}$ to prove it for systems of generalized transverse type, and systems of MHD type respectively. 
	Let 
	\begin{align}
	r^*_u(\nu)=\sup \{ r\in\mathbb R: u\in H^r\ \text{at}\ \nu\} \ \ \ \ \nu\in T^*X\setminus 0
	\end{align}
	be the regularity function.
	\begin{theorem}\label{main theorem}
		Let $P\in\Psi ^{m}_{phg}$ be an $N\times N$ system of generalized transverse type (or of MHD type) at $\nu_0\in \Sigma_2$, and let $A\in \Psi^0_{phg}$ be an $N\times N$ system such that the dimension of $\mathcal N_A\cap\mathcal N_{P}$ is equal to $1$ at $\nu_0$, and $M_A=\pi_1(\mathcal N_A\cap\mathcal N_{P}\setminus 0)$ is a hypersurface near $\nu_0$,  where $\pi_1:T^*X\times\mathbb C^N\rightarrow T^*X$ is the projection along the fibers. Assume that $u\in \mathcal D'(X,\mathbb C^{N})$ satisfies $\min (r^*_{Pu}+m-1,r^*_{Au})>r$ at $\nu_0$. Then, $\operatorname{Pol}^r(u)$ is a union of $\mathcal C^\infty$ line bundles in $\mathcal N_A\cap\mathcal N_{P}$ over bicharacteristics of $M_A=\pi_1(\mathcal N_A\cap\mathcal N_{P}\setminus 0)$ near $\nu_0$.
	\end{theorem}

	The plan of this paper is as follows: in Section $\ref{Previous results}$, we mention previous results on the propagation of polarization sets. More precisely, we will state Dencker's propagation result for systems of real principal type which was proven in \cite{denckerprincipaltype}, and Dencker's propagation result for systems of uniaxial type that was proven in \cite{denckerdoublerefraction}.  Also, we state Dencker's result for the propagation of polarization sets for systems of transverse type; see \cite{denckertransversaltype}. Note that in \cite{denckerdoublerefraction} and \cite{denckertransversaltype}, Dencker proved several results for the propagation of polarization sets under different conditions. Here we just mention the result which is similar to the result in our main theorem. In Sections $\ref{section Propagation of polarization sets for systems of generalized uniaxial type}$ and $\ref{section propagation of polarization sets for systems of MHD type}$, we define systems of generalized transverse type, and systems of MHD type, and we prove Theorem $\ref{main theorem}$ for both types of systems. In Section $\ref{section application}$, we give an application for the results in \cite{hansenlagrangiansolutions}, and for the propagation of polarization sets for systems of MHD type, so we divide it into two subsections. First, we give the set of equations describing the ideal MHD, and we linearize it. In Section $\ref{section application transport eq}$, we write the linearized ideal MHD equations in the form of a wave equation, and we give the characteristic variety of this wave equation under some assumptions which was calculated in \cite{lecturesmhd}. Then, we calculate the transport equation under these assumptions as an application to Hansen's and R\"{o}hrig's results in \cite{hansenlagrangiansolutions}. In Section $\ref{section application MHD type system}$, we return to the linearized ideal MHD equations, and we calculate the eigenvalues and their multiplicities which are not constant. Then, we study the propagation of polarization sets, where we observe different cases, some in which our system is of real principal type, some in which our system is of uniaxial type, and one where our system is of MHD type.\\~~\\
 \textit{Acknowledgments.}
 This work was part of my PhD thesis at Georg-August University, G\"{o}ttingen, supervised by Prof. Dr. Ingo Witt, and involved with the RTG 2491, check \cite{rayhanaphdthesis}. My PhD studies was funded by a DAAD scholarship.
	\section{Previous results}\label{Previous results}
 In this section, we state some previous results regarding the propagation of polarization sets. More precisely, we state the results for systems of real principal type, systems of uniaxial type, and systems of transverse type proven in \cite{denckerprincipaltype}, \cite{denckerdoublerefraction}, and \cite{denckertransversaltype} respectively. First, we will state the definition of polarization sets given by Dencker in \cite{denckerprincipaltype}.
 \begin{defn}
		For $u\in \mathcal D'(X,\mathbb C^N)$, the polarization set of $u$ is given by 
		\begin{align}
		\operatorname{Pol}(u)=\bigcap \mathcal N_B\subseteq (T^*X\setminus 0)\times \mathbb C^N,
		\end{align}
		where $\mathcal N_B=\ker \sigma(B)$, and the intersection is taken over those $1\times N$ systems $B\in \Psi^0_{phg}$ such that $Bu\in \mathcal C^\infty$.
	\end{defn}
 The $H^s$-polarization set, where $H^s$ is the usual Sobolev space is defined similarly.
	\begin{defn}
		For $u\in \mathcal D'(X,\mathbb C^N)$, the $H^s$-polarization set is given by 
		\begin{align}\label{Hs-polarization set}
		\operatorname{Pol}^s(u)=\bigcap \mathcal N_B\subseteq (T^*X\setminus 0)\times \mathbb C^N,
		\end{align}
		where $\mathcal N_B=\ker \sigma(B)$, and the intersection is taken over those $1\times N$ systems $B\in \Psi^0_{phg}$ such that $Bu\in H^s$.
	\end{defn}
 \subsection{Systems of real principal type}
	First, we want to state Dencker's result regarding the propagation of polarization sets for systems of real principal type. For the definition of real principal type, we will differentiate between two cases, the scalar case, and the case of system of pseudodifferential operators. 
 
 \begin{defn}\label{def real principal type scalar}
	We say that $P\in \Psi^m(X)$ is of real principal type if the principal symbol $\sigma(P)=p$ is real and the Hamilton field $H_{p}=\sum\partial_{\xi_j}p\partial_{x_j}-\partial_{x_j}p\partial_{\xi_j}$ is non-vanishing, and does not have the radial direction when $p=0$.
\end{defn}
\begin{defn}[Case of system of pseudodifferential operators]\label{def systems of real principal type}
	An $N\times N$ system $P$ of pseudodifferential operators on $X$ with principal symbol $p(x,\xi)$ is of real principal type at $(y,\eta)\in T^*X\setminus 0$ if there exists an $N\times N$ symbol $\tilde{p}(x,\xi)$ such that 
	\begin{align*}
	\tilde{p}(x,\xi)p(x,\xi)=q(x,\xi)\cdot\operatorname{Id}_N
	\end{align*}
	in a neighborhood of $(y,\eta)$ where $q(x,\xi)$ is a scalar symbol of real principal type and $\operatorname{Id}_N$ is the identity in $\mathbb C^N$.
\end{defn}
 Assume $P(x,D)$ to be an $N\times N$ system of classical pseudodifferential operators on an $n$-dimensional smooth manifold $X$ of order $m$. The symbol of $P$ is an asymptotic sum of homogeneous terms: $p(x,\xi)+p_{m-1}(x,\xi)+p_{m-2}(x,\xi)+...$ where $p$ is the principal symbol of $P$ and $p_j$ is homogeneous of degree $j$. Assume $P$ to be of real principal type, and let 
\begin{align}
\Sigma=\{(x,\xi): \det p(x,\xi)=0\}
\end{align}
be the characteristic set of $P$.
To state the result of the propagation of polarization set given by Dencker in \cite{denckerprincipaltype}, we have to introduce first the connection he defined, and give the definition of the Hamilton orbit. Let
\begin{align}
D_p w=H_q w+\frac{1}{2}\{\tilde{p},p\} w+i\tilde{p}p^s_{m-1} w,
\end{align}
where $w$ is $\mathcal C^\infty$ function on $T^*X\setminus 0$ with values in $\mathbb C^N$, $\{,\}$ is the Poisson bracket, that is $\{\tilde{p},p\}=H_{\tilde{p}}p$, and $p^s_{m-1}$ is the subprincipal of $P$ defined by 
\begin{align}
p^s_{m-1}=p_{m-1}-(2i)^{-1}\sum \partial _{x_j}\partial_{\xi_j} p.
\end{align}
$D_P$ is a connection on $\mathcal N_P$ over $\Sigma$, that is, if $w\in\ker p$ at one point of a bicharacteristic of $\Sigma$, then $D_P w\in\ker p$ along the bicharacteristic if and only if $w\in\ker p$ there. Hence, each parallel section (that is $w$ such that $D_P w=0$) is uniquely determined by one point. $D_P$ depends on the choice of $\tilde{p}$ and $q$, however Dencker showed that different choices of $\tilde{p}$ and $q$ only change the solution of $D_P w=0$ in $\mathcal N_P$ by a scalar factor. This motivated him to define the following. 
\begin{defn}[Hamilton Orbit]
	A Hamilton orbit of a system $P$ of real principal type is a line bundle $L\subseteq \mathcal N_P|_{\gamma}$, where $\gamma$ is an integral curve of the Hamilton field of $\Sigma$, and $L$ is spanned by $\mathcal C^\infty$ section $w$ satisfying $D_p w=0$.	
\end{defn}
\begin{theorem}[Dencker's propagation result]
	Let $P$ be an $N\times N$ system of pseudodifferential operators on a manifold $X$ and let $u\in\mathcal D'(X,\mathbb C^N)$. Assume that $P$ is of real principal type at $(y,\eta)\in\Sigma$, and that $(y,\eta)\notin \operatorname{WF}(Pu)$. Then, over a neighborhood of $(y,\eta)$ in $\Sigma$, $\operatorname{Pol}(u)$ is a union of Hamilton orbits of $P$.
\end{theorem}
	In \cite{Gerardsobolev}, G\'{e}rard stated that we have similar propagation result for the $H^s$-polarization sets  for systems of real principal type. 
	\begin{theorem}
		Let $P$ be an $N\times N$ system of pseudodifferential operators on a manifold $X$ of order $m$, and let $u\in\mathcal D'(X,\mathbb C^N)$. Assume that $P$ is of real principal type at $(y,\eta)\in\Sigma$, and that $(y,\eta)\notin \operatorname{WF}^s(Pu)$. Then over a neighborhood of $(y,\eta)\in\Sigma$, $\operatorname{Pol}^{s+m-1}(u)$ is a union of Hamilton orbits of $P$.
	\end{theorem}
 \subsection{Systems of uniaxial type}
	Let $P\in \Psi^{m}_{phg}(X)$ be an $N\times N$ system of classical pseudodifferential operators on a smooth manifold $X$. Let $p$ be the principal symbol of $P$.  Let 
$\Sigma=(\det p)^{-1}(0)$ be the characteristics of $P$, and let 
\begin{align}
\Sigma_2=\{(x,\xi)\in \Sigma: d(\operatorname{det}p)=0\ \text{at}\ (x,\xi)\},
\end{align}
and $\Sigma_1=\Sigma\setminus \Sigma_2$.
Assume that we have \begin{align}\label{2.1}
\begin{split}
&\Sigma=S_1\cup S_2,\ \text{where}\ S_1\ \text{and}\ S_2\ \text{are non-radial hypersurfaces tangent at}\\ &\Sigma_2=S_1\cap S_2\ \text{of exactly order}\ k_0\geq 1,
\end{split}
\end{align}
microlocally near $\nu_0\in\Sigma_2$. This means that the Hamilton field of $S_j$ does not have the radial direction $\langle\xi,\partial_\xi \rangle$, and it means also that the $k_0$th jets of $S_1$ and $S_2$ coincide on $\Sigma_2$, but no $(k_0+1)$th jet does. Note that we have $P$ is of real principal type at $\Sigma_1$, since $d(\det p)\neq 0$ there; see Definition $\ref{def systems of real principal type}$.
$(\ref{2.1})$  gives us that $\Sigma_2$ has to be a manifold of codimension $\geq2$. We assume that 
\begin{align}\label{2.2}
\Sigma_2\ \text{is an involutive manifold of codimension}\ d_0\geq 2.
\end{align}
Moreover, we assume that
\begin{align}\label{2.3}
\text{the (complex) dimension of the fiber of}\  \mathcal N_p \ \text{is equal to}\ 2\ \text{at}\ \Sigma_2,
\end{align}
and
\begin{align}\label{2.4}
d^2(\det p)\neq 0\ \text{at}\ \Sigma_2,
\end{align}
that is, $p$ vanishes of first order on the kernel.
We want to consider the limits of $\mathcal N_P|_{\Sigma_1}$ when we approach $\Sigma_2$, so let
\begin{align}\label{Np at Sj}
\mathcal N_p^j=\mathcal N_p|_{S_j\setminus\Sigma_2},
\end{align}
$T_{\Sigma_2}\Sigma:=T_{\Sigma_2}S_1=T_{\Sigma_2}S_2$ (note here $\Sigma$ is not a manifold), and $\partial\Sigma_1:=T_{\Sigma_2}\Sigma/ T\Sigma_2$. Here $\partial\Sigma_1$ is the normal bundle of $\Sigma_2$ in $S_1$ which is equal to the normal bundle of $\Sigma_2$ in $S_2$. Let $i_0:\Sigma_2\rightarrow \partial \Sigma_1$ denotes the zero section of $\partial\Sigma_1$. By the tubular neighborhood theorem we know that there exists a diffeomorphism $\Phi$ from some neighborhood $\mathcal U\subset S_j$ of $\Sigma_2$ to a neighborhood $\mathcal U_0 \in\partial \Sigma_1$ of the zero section of $\partial\Sigma_1$, and $\Phi$ identifies $\Sigma_2$ itself with the zero section.

Before giving the definition of systems of uniaxial type, we need to give the definition of the limit polarizations.
\begin{defn}
	For $j=1,2$, we define the limit polarizations 
	\begin{align}
	\partial\mathcal N_P^j=\{(\nu,\rho,z)\in\partial \Sigma_1\times \mathbb C^N: \rho\neq 0 \ \text{and}\ z=\underset{k\rightarrow\infty}{\lim} z_k\},
	\end{align}
	where $z_k\in\ker p(\nu_k)$ and $\nu_k\in S_j\setminus \Sigma_2$ satisfy $(\nu-\nu_k)/\lvert \nu-\nu_k\rvert\rightarrow\rho/\lvert\rho\rvert$ when $k\rightarrow\infty$.
\end{defn}
 $\partial \mathcal N_P^j$ is conical in $\xi$ and $\rho$, and homogeneous in the fiber, but it may have (complex) dimension $>1$ at $(\nu,\rho)$. We assume that the fiber of 
\begin{align}\label{2.7}
\partial \mathcal N_P^1\cap\partial\mathcal N_P^2=\{0\}\ \text{over} \  \partial\Sigma_1\setminus(\Sigma_2\times 0).
\end{align}
This condition means that no element in $\mathcal N_P|_{\Sigma_2}$ can be the limit of polarization vectors on both characteristic surfaces, along the same direction. Dencker showed that $(\ref{2.7})$ implies that $\partial  \mathcal N_P^j$ is a complex line bundle over $\partial \Sigma_1\setminus(\Sigma_2\times 0)$, if we assume $(\ref{2.1})$-$(\ref{2.4})$.
Now, we give the definition of systems of uniaxial type.
\begin{defn}
	The system $P$ is of uniaxial type at $\nu_0\in\Sigma_2$, if $(\ref{2.1})$-$(\ref{2.4})$ and $(\ref{2.7})$ hold microlocally near $\nu_0$.
\end{defn}
  If $P\in \Psi^{m}_{phg}$ is of uniaxial type and $Pu\in H^{r}$ near $\nu\in \Sigma_1$, then we already know the result as Dencker showed in \cite{denckerprincipaltype} that $\operatorname{Pol}^{r+m-1}(u)$ is a union of Hamilton orbits in $\mathcal N_P$ near $\nu$ because $P$ is of real principal type at $\Sigma_1$. Now, we want to give Dencker's result  for the propagation of the polarization set when we approach $\Sigma_2$. Here $S_1$ and $S_2$ are tangent at $\Sigma_2$, so their Hamilton fields are parallel on $\Sigma_2$. Since $\Sigma_2$ is involutive, the Hamilton fields are tangent to $\Sigma_2$. Therefore, $\Sigma$ and $\Sigma_2$ are foliated by the bicharacteristics of $\Sigma$. Also, Dencker proved that $\partial\Sigma_1\setminus (\Sigma_2\times 0)$ is foliated by limit bicharactersitcs, which are liftings of bicharactersitics in $\Sigma_2$, and that $\partial\mathcal N_P^1\cup \partial\mathcal N_P^2$ is foliated by limit Hamilton orbits, which are liftings of limits of Hamilton orbits, and are unique line bundles over limit bicharacteristics.
	\begin{theorem}
		Let $P\in\Psi ^{m}_{phg}$ be an $N\times N$ system of uniaxial type at $\nu_0\in \Sigma_2$, and let $A\in \Psi^0_{phg}$ be a $1\times N$ system such that the dimension of $\mathcal N_A\cap\mathcal N_P$ is equal to $1$ at $\nu_0$. Assume that $u\in \mathcal D'(X,\mathbb C^N)$ satisfies $\min (r^*_{Pu}+m-1,r^*_{Au})>r$ at $\nu_0$. Then, $\operatorname{Pol}^r(u)$ is a union of $\mathcal C^\infty$ line bundles in $\mathcal N_A\cap \mathcal N_P$ over bicharacteristics of $\Sigma$ in $M_A=\pi_1(\mathcal N_A\cap\mathcal N_P\setminus 0)$ near $\nu_0$, where $\pi_1:T^*X\setminus 0\times\mathbb C^N\rightarrow T^*X$ is the projection along the fibers.
	\end{theorem}
 Moreover, in \cite{denckerdoublerefraction}, Dencker showed under what assumptions we get $\operatorname{Pol}^r(u)$ is union of limits of Hamilton orbits in $\mathcal N_A\cap\mathcal N_P$ near $\nu_0\in\Sigma_2$. 
 \subsection{Systems of transverse type}\label{subsection systems of transverse type}
 Finally, we want to state Dencker's result regarding the propagation of polarization sets for systems of transverse type; see \cite{denckertransversaltype}. Let $P\in\Psi ^{m}_{phg}$ be an $N\times N$ system of classical pseudodifferential operators on a smooth manifold $X$, $p=\sigma(P)$ be the principal symbol, and $\Sigma=(\det p)^{-1}(0)$ be the characteristics of $P$. Let 
 \begin{align}
     \Sigma_2=\{(x,\xi)\in\Sigma: d(\det p)=0\ \text{at}\ (x,\xi)\},
 \end{align}
 and $\Sigma_1=\Sigma\setminus \Sigma_2$. For systems of transverse type we have $\Sigma$ is a union of two non radial hypersurfaces intersecting transversally at $\Sigma_2$. More precisely, the systems of transverse type is defined as the following
 \begin{defn}
    The system $P$ is of transverse type at $\nu_0\in \Sigma_2$ if
    \begin{align}
        &\Sigma_2 \ \text{is a non-radial involutive manifold of codimension}\  2,\\
       & \det p=e \cdot q, \ \text{where}\ e\neq 0\ \text{and}\ q\ \text{is real valued with Hessian having rank}\  2\ \text{and positivity}\ 1,\\
       &\dim \ker p=2\ \text{on}\ \Sigma_2,
    \end{align}
    microlocally near $\nu_0$.
 \end{defn}
 Similar to the case of systems of uniaxial type, if $P\in\Psi^m_{phg}$ is of transverse type and $Pu\in H^r$ near $\nu\in\Sigma_1$, then $P$ is of real principal type at $\nu$. Let $\mathcal N_P^j$ be as in $(\ref{Np at Sj})$. In \cite{denckertransversaltype}, Dencker modified slightly the definition of limit polarizations. 
 \begin{defn}
    For $j=1, 2$, the limit polarizations is defined by 
    \begin{align}
    \partial\mathcal N_P^j=\{(\nu,z)\in\Sigma_2\times\mathbb C^N: z=\underset{k\rightarrow\infty}{\lim}z_k\}, 
    \end{align}
where $z_k\in\ker p(\nu_k)$ and $S_j\setminus \Sigma_2\ni \nu_k\rightarrow \nu$.
 \end{defn}
 $\partial \mathcal N_P^j$ is conical in $\xi$ and linear in the fibers. Dencker showed that $\partial\mathcal N_P^j$ is a $\mathcal C^\infty$ line bundle over $\Sigma_2$, $j=1,2$, and that
 \begin{align}
    \partial\mathcal N_P^1\cap \partial\mathcal N_P^2=\{0\} \ \ \ \text{over}\ \Sigma_2.
 \end{align}
 Here, $S_1$ and $S_2$ are transverse at $\Sigma_2$, so their Hamilton fields are non-parallel on $\Sigma_2$. Since $\Sigma_2$ is involutive of codimension $2$, the Hamilton fields of $S_j$ are tangent to $\Sigma_2$ and generate the two-dimensional foliation of $\Sigma_2$. Moreover,
 $\partial\mathcal N_P^j$ is foliated by limit Hamilton orbits which are limits of Hamilton orbits in $\mathcal N_P^j$, and are unique line bundles over bicharacteristics in $S_j$ at $\Sigma_2$ for $j=1,2$. 
 \begin{theorem}
     Let $P\in\Psi ^{m}_{phg}$ be an $N\times N$ system of transverse type at $\nu_0\in \Sigma_2$, and let $A\in \Psi^0_{phg}$ be a $1\times N$ system such that the dimension of $\mathcal N_A\cap\mathcal N_P$ is equal to $1$ at $\nu_0$, and $M_A=\pi_1(\mathcal N_A\cap\mathcal N_P\setminus 0)$ is a hypersurface near $\nu_0$. Assume that $u\in\mathcal D'(X,\mathbb C^N)$ such that $Pu\in H^{r-m+1}$ and $Au\in H^r$ at $\nu_0$. Then $\operatorname{Pol}^r(u)$ is a union of (limit) Hamilton orbits in $\mathcal N_A\cap\mathcal N_P$. Here $\pi_1:T^*X\times\mathbb C^N\rightarrow T^*X$ is the projection along the fibers. 
 \end{theorem}
 Note that in this case $M_A=S_j$ for some $j$, and $\mathcal N_A\cap\mathcal N_P$ is a union of (limit) Hamilton orbits.
	\section{Propagation of polarization sets for systems of generalized transverse type}\label{section Propagation of polarization sets for systems of generalized uniaxial type}
	In this section, we generalize Dencker's result stated in Section $\ref{subsection systems of transverse type}$ by considering the system to have its characteristic set is union of $r_0$ hypersurfaces intersecting transversally at an involutive manifold of codimension $d_0\geq 2$, with $r_0\geq 2$. Let $P\in\Psi^{m}_{phg}(X)$ be an $N\times N$ system of classical pseudodifferential operators on a smooth manifold $X$. Let $p=\sigma(P)$ be the principal symbol and $\Sigma=(\det p)^{-1}(0)$ the characteristic set. Assume microlocally near $(x_0,\xi_0)\in\Sigma$ that
	\begin{align}\label{2.1 double refraction}
	\begin{split}
	\Sigma=&\bigcup_{j=1}^{r_0} S_j,\  r_0\geq2, \ \text{where}\  S_j\  \text{are non-radial hypersurfaces intersecting transversally at}\\ \  &\Sigma_2=\bigcap_{j=1}^{r_0} S_j,
	\end{split}
	\end{align}
	\begin{align}\label{2.2 double refraction}
	\Sigma_2\ \text{is an involutive manifold of codimension}\ d_0\geq 2.
	\end{align}
	Moreover, assume that
	\begin{align}\label{2.3 double refraction}
	\text{the dimension of the fiber of}\ \mathcal N_P\ \text{is equal to}\ r_0\ \text{at}\ \Sigma_2,
	\end{align}
	and
	\begin{align}\label{2.4 double refraction}
	d^{i}(\det p)=0\ \text{for}\ i<r_0 \ \text{and}\ d^{r_0}(\det p)\neq 0\ \ \ \ \text{at}\ \Sigma _2.
	\end{align}
 We want to consider the limits of $\mathcal N_P|_{\Sigma_1}$ when we approach $\Sigma_2$, so let
$T_{\Sigma_2}\Sigma:=\bigcup_{j=1}^{r_0}T_{\Sigma_2}S_j$ (note that $\Sigma$ is not a manifold), and $\partial\Sigma_1:=T_{\Sigma_2}\Sigma/ T\Sigma_2$. Before giving the definition of systems of generalized transverse type, we need to give the definition of the limit polarizations.
\begin{defn}\label{definition limit pol}
	For $j=1,..., r_0$, we define the limit polarizations 
	\begin{align}
	\partial\mathcal N_p^j=\{(\nu,\rho,z)\in\partial \Sigma_1\times \mathbb C^N: \rho\neq 0 \ \text{and}\ z=\underset{k\rightarrow\infty}{\lim} z_k\},
	\end{align}
	where $z_k\in\ker p(\nu_k)$ and $\nu_k\in S_j\setminus \Sigma_2$ satisfy $(\nu-\nu_k)/\lvert \nu-\nu_k\rvert\rightarrow\rho/\lvert\rho\rvert$ when $k\rightarrow\infty$.
\end{defn}
 $\partial \mathcal N_P^j$ is conical in $\xi$ and $\rho$, and homogeneous in the fiber.
	  We will assume that the fiber of 
\begin{align}\label{2.7 double refraction}
\mathcal \partial N_P^1\cap...\cap\mathcal \partial N_P^{r_0}=\{0\}\ \ \ \ \text{over}\ \ \partial \Sigma_1\setminus (\Sigma _2\times 0).
\end{align}
This condition means that no element in $\mathcal N_P|_{\Sigma_2}$ can be the limit of polarization vectors on all characteristic surfaces along the same direction.
	\begin{defn}
		The system $P$ is of generalized transverse type at $\nu_0\in \Sigma _2$, if $(\ref{2.1 double refraction})$-$(\ref{2.4 double refraction})$, and $(\ref{2.7 double refraction})$  hold microlocally near $\nu_0$. 
	\end{defn}
	\begin{prop}\label{prop1}
		Let $P\in \Psi ^1_{phg}$ be an $N\times N$ system of generalized transverse type at $\nu_0\in \Sigma _2$. Then by choosing suitable symplectic coordinates, we may assume that $X=\mathbb R\times \mathbb R^{n-1}$, $\nu_0=(0;(0,...,1))$, and 
		\begin{align}\label{2.3 singularities}
		S_j=\{(t,x;\tau,\xi)\in T^*(\mathbb R\times \mathbb R^{n-1}): \tau+\beta_j(t,x,\xi)=0\},\ \ j=1,...,r_0,
		\end{align}
		microlocally near $\nu_0$. Here $\beta_j$ are real and homogeneous of degree $1$ in $\xi$; with
		$\beta_1\equiv 0$, satisfies in a conical neighborhood of $\nu_0$
		\begin{align}\label{betaj}
		c\lvert \xi '\rvert\leq \lvert \beta_j\rvert\leq C\lvert \xi '\rvert,\ \ \ j=2,...,r_0, \ \ 0<c<C,
		\end{align}
		where $(\tau,\xi',\xi'')\in\mathbb R\times\mathbb R^{d_0-1}\times \mathbb R^{n-d_0}$, which gives $\Sigma _2=\{\tau=0, \xi'=0\}$. By conjugating $P$ with an elliptic, scalar Fourier integral operators, and multiplying with elliptic $N\times N$ systems of order $0$, we may assume that
		\begin{align}
		P\cong\begin{pmatrix}
		F & 0\\
		0 & E
		\end{pmatrix}\ \ \operatorname{mod}\ \mathcal C^\infty,
		\end{align}
		microlocally near $\nu_0$, where $E\in \Psi ^1_{phg}$ is an elliptic $(N-r_0)\times (N-r_0)$ system and 
		\begin{align}\label{3.4}
		F\cong \operatorname{Id}_{r_0} D_t+K(t,x,D_x)\ \ \operatorname{mod}\ \mathcal C^\infty.
		\end{align}
		Here $K(t,x,D_x)\in\mathcal C^\infty(\mathbb R,\Psi ^1_{phg})$ is an $r_0\times r_0$ system, such that $k=\sigma(K)$ has $0$, $\beta_2$,...,$\beta_{r_0}$ as eigenvalues.
	\end{prop}
	\begin{proof}We will prove it in a way similar to how Dencker proved the normal form for systems of uniaxial type. Since the result is local, we may assume $X=\mathbb R^n$. Because $\Sigma_2$ is involutive, we may choose symplectic, homogeneous coordinates $(x,\xi)\in T^*\mathbb R^n$ near $\nu_0\in \Sigma _2$, so that $\nu_0=(0;(0,...,1))$ and 
		\begin{align}
		\Sigma _2=\{(x,\xi)\in T^*\mathbb R^n:\xi'=0\},
		\end{align}
		where $\xi=(\xi',\xi'')\in\mathbb R^{d_0}\times\mathbb R^{n-d_0}$. We may also assume
		\begin{align}
		S_1=\{(x,\xi)\in T^*\mathbb R^n:\xi_1=0\},
		\end{align}
		near $\nu_0$.
		Now, we rename $x_1=t$, $(x_2,..,x_{d_0})=x'$, and $(x_{d_0+1},...,x_n)=x''$. Since $S_j$ is intersecting transversally with $S_1$ at $\Sigma _2$, we obtain
		\begin{align}
		S_j=\{(t,x;\tau,\xi)\in T^*(\mathbb R\times \mathbb R^{n-1}): \tau+\beta_j(t,x,\xi)=0\},
		\end{align}
		with $\beta_j$ real and homogeneous of degree $1$ in $\xi$, $\beta_1\equiv 0$, and 
		\begin{align}
		c\lvert \xi '\rvert\leq \lvert \beta_j-\beta_k\rvert\leq C\lvert \xi '\rvert, \ \ \ j\neq k,
		\end{align}
		in a conical neighborhood of $\nu_0$. By taking $k=1$, we obtain 
		\begin{align}
		c\lvert \xi '\rvert\leq \lvert \beta_j\rvert\leq C\lvert \xi '\rvert, \ \ \ j=2,...,r_0.
		\end{align}
		Using that $\dim \mathcal N_P=r_0$ at $\Sigma_2$,	we can find an $N\times N$ elliptic matrix $b$ homogeneous of degree $0$ in the $\xi$ variables which maps $\operatorname{Im}p$ to $\{ z\in\mathbb C^{N}; z_j=0, \ j\leq r_0\}$ over $\Sigma_2$ near $\nu_0$, and we can choose an $N\times N$ matrix $a$ homogeneous of degree $0$ in the $\xi$ variables such that $a^{-1}$ maps $\ker p$ onto $\{z\in\mathbb C^{N}; z_j=0,\ j>r_0\}$ over $\Sigma_2$ near $\nu_0$. 
		Then we have 
		\begin{align}\label{matrix principal symbol}
		bpa=
		\begin{pmatrix}
		s_{11} & s_{12}
		\\ s_{21} & e
		\end{pmatrix}
		\end{align}
		such that $e$ is an $(N-r_0)\times (N-r_0)$ matrix which is elliptic at $\nu_0$, and $s_{11}$, $s_{12}$, $s_{21}$, vanish on $\Sigma_2$ near $\nu_0$.
		
		Now, we choose $N\times N$ systems of pseudodifferential operators $A$ and $B$ with principal symbols $a$, and $b$ respectively. Then 
		\begin{align}
		BPA=\begin{pmatrix}
		S_{11} &S_{12}\\
		S_{21}& E
		\end{pmatrix}
		\end{align}
		where its principal symbol is given by $(\ref{matrix principal symbol})$. As $E$ is a system of order $1$ which is elliptic at $\Sigma_2$, choose $J$ to be its microlocal parametrix of order $-1$. Multiply $BPA$  from the left with 
		\begin{align}
		B_1=
		\begin{pmatrix}
		\operatorname{Id}_{r_0}& -S_{12}J\\
		0& \operatorname{Id}_{N-r_0}
		\end{pmatrix}.
		\end{align}
		Multiply also $B_1BPA$ from the right by 
		\begin{align}
		A_1=
		\begin{pmatrix}
		\operatorname{Id}_{r_0}& 0\\
		-JS_{21}& \operatorname{Id}_{N-r_0}
		\end{pmatrix}.
		\end{align}
		Hence, we get 
		\begin{align}
		P\cong\begin{pmatrix}
		F & 0\\
		0 & E
		\end{pmatrix}\ \ \operatorname{mod}\ \mathcal C^\infty,
		\end{align}
		microlocally near $\nu_0$, where $E\in \Psi ^1_{phg}$ is an elliptic $(N-r_0)\times (N-r_0)$ system. If $f$ is the principal symbol for $F$, then conditions $(\ref{2.4 double refraction})$ and $(\ref{2.3 singularities})$ imply 
		\begin{align}
		\det f =c\tau\prod_{i=2}^{r_0}(\tau+\beta_i), \ 0\neq c\in  S^{-1},
		\end{align}
		thus $\partial_{\tau}^{r_0}(\det f)=\det(\partial_{\tau} f)\neq 0$ at $\Sigma_2$. By Theorem A.3 in \cite{denckerdoublerefraction}, and homogeneity, we may find homogeneous system $C_0\in S^0$ such that 
		\begin{align}
		f=C_0(\tau\operatorname{Id}_{r_0}+k(t,x,\xi)),
		\end{align}
		where $\det C_0\ne 0$ at $\Sigma_2$. By multiplication with an elliptic system, we may assume $C_0\equiv\operatorname{Id}_{r_0}$. Thus, $\det f=\tau\prod_{i=2}^{r_0} (\tau+\beta_j)$, which implies that $k(t,x,\xi)$ has the eigenvalues $0$, $\beta_2$,...,$\beta_{r_0}$. If $f_0\in S^0$ is the term homogeneous of degree $0$ in the expansion of $F$, then Theorem A.4 in \cite{denckerdoublerefraction}, and homogeneity give
		\begin{align}
		f_0=B_{-1}f+B_0,
		\end{align} 
		where $B_0\in\mathcal C^\infty (\mathbb R, S^0)$ is independent of $\tau$, and $B_{-1}\in  S^{-1}$. By multiplying $f$ with an operator with symbol $\operatorname{Id}_{r_0}-B_{-1}$, we may assume $B_{-1}\equiv 0$. By induction over lower order terms we obtain $(\ref{3.4})$.
	\end{proof}
We want to introduce symbol classes adapted to the functions $\beta_j$ defined in $(\ref{betaj})$ for $j=2,...,r_0$. Let 
\begin{align}\label{w}
\vartheta(\xi)=\langle \xi'\rangle,
\end{align}
where $\langle \xi' \rangle=(1+\lvert \xi'\rvert^2)^{1/2}$, thus $\vartheta\approx 1+\lvert\beta_j\rvert$. Consider the metric 
\begin{align}\label{metric generalized uniaxial}
g(dx,d\xi)=\lvert dx\rvert ^2+\lvert d\xi'\rvert^2\langle\xi'\rangle^{-2}+\lvert d\xi''\rvert^2 \langle \xi \rangle^{-2},
\end{align}
and $h^2=\sup g/g^\sigma=\langle \xi' \rangle^{-2}$. We get that $g$ is $\sigma$ temperate, and $\beta_j\in S(\vartheta,g)$. Check \cite[Chapter XVIII] {hormander3} to know more about the symbol classes $S(\vartheta,g)$ of the Weyl calculus. Moreover, using Taylor's formula we can write 
\begin{align}\label{taylor betaj generalized}
\beta_j=a_j \xi',
\end{align}
with $a_j\in S^0$ is homogeneous of degree $0$ in $\xi$.
\begin{prop}\label{prop1 generalized}
	Let 
	\begin{align}\label{form normal}
	P=\operatorname{Id}_{r_0} D_t+K(t,x,D_x)
	\end{align}
	be an $r_0\times r_0$ system with $K\in\mathcal C^\infty (\mathbb R,\Psi^1_{phg})$, such that the eigenvalues of $k=\sigma(K)$ are $0$, $\beta_2$,...,$\beta_{r_0}$. Then $P$ is of generalized transverse type if and only if $k\in\mathcal C^\infty(\mathbb R, S(\vartheta,g))$.
\end{prop}
\begin{proof}
	First, let $k=(k_{ij})_{1\leq i,j\leq r_0}$. Let
$\alpha=(\alpha_1,...,\alpha_{1r_0})\in\mathcal C^\infty(\mathbb R,  S^1)$, with $\alpha_i=(k_{i1},...,k_{i r_0})$ for $i=1,...,r_0$ homogeneous of degree $1$ in $\xi$. By homogeneity,
	\begin{align}
	k\in\mathcal C^\infty(\mathbb R, S(\vartheta,g)) \  \Leftrightarrow\alpha=O(\beta_j),
	\end{align}
	for every $j=2,...,r_0$.

	Assume that $\alpha=O(\beta_j)$ and $(\nu,\rho)\in\partial \Sigma _1$, $\rho\neq 0$. Choose $\Sigma\setminus \Sigma_2\ni \nu_l\rightarrow \nu$ such that $(\nu-\nu_l)\lvert \nu-\nu_l\rvert^{-1}\rightarrow \rho/\lvert \rho\rvert$, $l\rightarrow\infty$. 
	Let us define
	\begin{align}\label{gamma1 generalized}
	\gamma_s^{ij}(\nu,\rho):=\underset{\nu_l\rightarrow \nu}{\lim} \frac{k_{ij}}{\beta_s}(\nu_l)\ \  \text{for}\ \ s\in\{2,...,r_0\}\ \  \text{and}\ \ 1\leq i,j\leq r_0. 
	\end{align}
    Since $\alpha=O(\beta_j)$ does not depend on $\tau$, so the above definition is independent of the choice of $\nu_l$.
	We get 
	\begin{align}
	\partial\mathcal N_{P}^1(\nu,\rho)=\ker((\gamma_s^{ij}(\nu,\rho))_{1\leq i,j\leq r_0}),\ \forall s\in\{2,...,r_0\},
	\end{align}
	where $(\gamma_s^{ij}(\nu,\rho))_{1\leq i,j\leq r_0}$ denote the matrix with entries $\gamma_s^{ij}(\nu,\rho)$ for $1\leq i,j\leq r_0$.
	\begin{align}
	\partial \mathcal N_{P}^s(\nu,\rho)=\ker(-\operatorname{Id}_{r_0}+(\gamma_s^{ij}(\nu,\rho))_{1\leq i,j\leq r_0})\ \text{for} \ s\in\{2,...,r_0\}.
	\end{align}
	It is easy to see that the condition $(\ref{2.7 double refraction})$ is satisfied.
	
	On the other hand, assume that $\alpha\neq O(\beta_j)$ at $\nu\in\Sigma _2$. Then, there exists a sequence $\nu_l=(t_l,x_l;0,\xi_l)\rightarrow \nu$, such that 
	\begin{align}
	\lvert \alpha(\nu_l)\rvert >l \lvert \beta_j(\nu_l)\rvert,  \ \ \forall l\in\mathbb N.
	\end{align}
	It is no restriction to assume that $\{(\nu-\nu_l)\lvert \nu-\nu_l\rvert^{-1}\}$ has a limit $0\neq \rho\in\partial \Sigma _1|_{\nu}$ as $l\rightarrow \infty$, and that 
	\begin{align}
	\varepsilon^{i,j}=\underset{l\rightarrow\infty}{\lim}k_{ij}(\nu_l)/\lvert\alpha(\nu_l)\rvert \ \ \text{exists}.
	\end{align}
	Since $\beta_j(\nu_l)/\lvert\alpha(\nu_l)\rvert\rightarrow 0$, we get that
	\begin{align}
	\partial\mathcal N_{P}^s(\nu,\rho)\supseteq \ker((\varepsilon^{i,j})_{1\leq i,j\leq r_0})\ \ \ \text{for}\ s=1,...,r_0.
	\end{align}
	Now, we want to show that $\ker((\varepsilon^{i,j})_{1\leq i,j\leq r_0})\neq\{0\}$. 
	Since, we have $\det k=0$, we get $\det \big((\varepsilon^{i,j})_{1\leq i,j\leq r_0}\big)=0$. Hence, we get that the rank of $(\varepsilon^{i,j})_{1\leq i,j\leq r_0}$ is less than or equal to $r_0-1$, which gives in turn that $\dim \ker ((\varepsilon^{i,j})_{1\leq i,j\leq r_0})$ is greater than or equal to $1$.
\end{proof}

 We will introduce two spaces that we are going to use. These spaces were also given by Dencker in \cite{denckertransversaltype} and he showed the relation between these two spaces by a lemma that we will also state.
Let $H^{r,s}$ be the space of $u\in\mathcal S'$ (here $u$ depends on $t$ and $x$) satisfying
\begin{align}\label{3.29 double}
\lVert u\rVert ^2_{r,s}=(2\pi)^{-n}\int\lvert \hat{u}(\tau,\xi)\lvert ^2\langle(\tau,\xi)\rangle^{2r}\langle(\tau,\xi')\rangle^{2s} d\tau d\xi<\infty.
\end{align}
We say that $u\in H^{r,s}$ at $\nu\in T^*\mathbb R^n\setminus 0$, that is, $\nu\notin\operatorname{WF}^{r,s}(u)$, if $u=u_1+u_2$, where $u_1\in H^{r,s}$ and $\nu\notin \operatorname{WF}(u_2)$. 

Similarly, when we have $u$ depends only on $x$ then the norm becomes
\begin{align}
\lVert u\rVert ^2_{r,s}=(2\pi)^{-n}\int\lvert \hat{u}(\xi)\lvert ^2\langle\xi\rangle^{2r}\langle \xi'\rangle^{2s} d\xi<\infty.
\end{align}
We have $S^0\subset S(1,g)$. Note that the spaces $H^{r,s}$ is a particular case of the spaces $B_{p,k}$ introduced by H\"{o}rmander; check \cite{hormander2}, where $p=2$ and $k(\tau,\xi)=\langle \tau,\xi \rangle^r \langle \tau,\xi'\rangle^s$.
	\begin{prop}\label{prop 3.7 generalized}
		Assume that $P$ is an $r_0\times r_0$ system of pseudodifferential operators of order $1$ on $\mathbb R^n$, on the form $(\ref{form normal})$, with $K\in \mathcal C^\infty(\mathbb R, \operatorname{Op} (S(\vartheta,g)))$  near $\nu_0\in \Sigma _2$. Let $u\in\mathcal S'(\mathbb R^n,\mathbb C^{r_0})$ and assume $Pu\in H^{r,s}$ at $\nu_0$. Then, for every $\delta>0$ we can find $c_{\delta}$ and $C_{\delta,N}>0$ and $v_\delta\in H^{r,s+1}$ at $\nu_0$, such that $u_\delta=u-v_\delta$ satisfies
		\begin{align}\label{3.30}
		\lvert \hat {u}_\delta(\tau,\xi)\rvert \leq C_{\delta,N}\langle(\tau,\xi)\rangle^{-N}, \ \ \ \forall N,
		\end{align}
		when $\lvert\tau\rvert >c_\delta(\langle\xi \rangle^\delta+\langle \xi' \rangle)$.
	\end{prop}
	\begin{proof}
		Follow the proof of Proposition $A.1$ in \cite{denckertransversaltype}. 
	\end{proof}
Let $H_{*}^{r,s}$ be the Banach space of $u\in \mathcal S'$, satisfying 
\begin{align}
(\lVert u\rVert^*_{r,s})^2=(2\pi)^{-n}\int\lvert \hat {u}(\tau,\xi)\lvert^2\langle\xi \rangle^{2r}\langle \xi'\rangle^{2s} d\tau d\xi<\infty.
\end{align}
If we have $u$ depends only on $x$ then the norm becomes 
\begin{align}
(\lVert u\rVert^*_{r,s})^2=(2\pi)^{-n}\int\lvert \hat {u}(\xi)\lvert^2\langle\xi \rangle^{2r}\langle \xi'\rangle^{2s}  d\xi<\infty.
\end{align}
Hence, $\lVert u\rVert^*_{r,s}=\lVert u\rVert_{r,s}$ when $u$ depends only on $x$.

If $u\in H^*_{r,s}$ then we get $u|_{t=\rho}\in H^{r,s}$ for almost all $\rho$, by Fubini's theorem. If $u\in\mathcal S'$ satisfies $(\ref{3.30})$, then 
\begin{align}
\lVert u\rVert^*_{r-\delta s_{-},s}\leq C_{r,s}(\lVert u \rVert_{r,s}+1)\leq C'_{r,s}(\lVert u\rVert^*_{r+\delta s_{+},s}+1),\ \forall r,s\in\mathbb R,
\end{align}
where $s_{\pm}=\max(\pm s,0)$. Hence, we lose only $O(\delta)$ derivatives when taking restriction to $\{t=r\}$, for almost all $r$.
\begin{defn}
	Let $u\in\mathcal S'(\mathbb R^n)$, and assume $\xi\neq 0$ in $\operatorname{WF}(u)$. We say that $u\in H_*^{r,s}$ at $(t_0,x_0,\xi_0)$, that is, $(t_0,x,_0,\xi_0)\notin \operatorname{WF}_*^{r,s}(u)$, if there exists $\phi(t,x,\xi)\in\mathcal C^\infty(\mathbb R, S^0_{1,0})$ such that $\phi(t,x,D_x)u\in H_*^{r,s}$ and $\underline{\lim}_{\lambda\rightarrow\infty}\lvert \phi(t_0,x_0,\lambda \xi_0)\rvert\neq 0$.
\end{defn}
We have 
\begin{align}\label{3.35}
(t_0,x_0,\xi_0)\notin \operatorname{WF}_*^{r,s}(u)\Rightarrow (x_0,\xi_0)\notin\operatorname{WF}^{r,s}(u_\rho),
\end{align}
for almost all $\rho$ close to $t_0$, where $u_\rho=u|_{t=\rho}$. If $\xi\neq 0$ in $\operatorname{WF}(u)$, then from \cite[Lemma\  2.3]{denckersingprincipaltype}, we get that 
\begin{align}
\pi_0(\operatorname{WF}^{r,0})(u)=\operatorname{WF}_*^{r,0}(u),
\end{align}
where $\pi_0(t,x;\tau,\xi)=(t,x,\xi)$. The following lemma gives the result for the more general wavefront sets:
\begin{lemma}\label{lemma 3.9 uniaxial}
	Assume that $u\in \mathcal S'(\mathbb R^n)$ satisfies $(\ref{3.30})$. Then $Au$ satisfies $(\ref{3.30})$, for any $A\in\mathcal C^\infty(\mathbb R, \Psi^m_{\delta,0})$ $\forall\  m$ and $\forall \ \delta>0$. We also obtain 
	\begin{align}
	\operatorname{WF}_*^{r-\delta s_{-},s}(u)\subseteq \pi_0(\operatorname{WF}^{r,s}(u))\subseteq \operatorname{WF}_*^{r+\delta s_{+},s}(u),
	\end{align}
	where $s_{\pm}=\max(\pm s,0)$ and $\pi_0(t,x;\tau,\xi)=(t,x,\xi)$. Since $u\in\mathcal C^\infty$ in $\pi_0^{-1}(\Sigma_2)\setminus \Sigma_2$ by $(\ref{3.30})$, we find 
	\begin{align}\label{3.37 uniaxial}
	\iota_0 (\operatorname{WF}_*^{r-\delta s_{-},s}(u))\subseteq \operatorname{WF}^{r,s}(u)\subseteq \iota_0(\operatorname{WF}_*^{r+\delta s_{+},s}(u))\ \ \ \ \text{on}\  \Sigma_2,
	\end{align}
	where $\iota_0(t,x,\xi)=(t,x;0,\xi)$.
\end{lemma}
\begin{proof}
	See \cite[Lemma A.3]{denckertransversaltype}.
\end{proof}
Note that $H^{r}=H^{r,0}$ is the usual Sobolev space.
Changing the notation, let $x_1=t$, $x'=(x_2,...,x_{d_0})$, then $x=(x_1,x',x'')\in\mathbb R\times \mathbb R^{d_0-1}\times \mathbb R^{n-d_0}$. Introduce the symbol classes $S^{r,s}=S(\langle \xi \rangle^r h^{-s},g)$ where $h^{-2}=1+\lvert\xi_1\rvert^2+\lvert\xi'\rvert^{2}$ and $\langle\xi \rangle$ are weights for the metric $g$ defined by 
\begin{align}
g_{x,\xi}(dx,d\xi)=\lvert dx\rvert ^2+\lvert d\xi\rvert ^2 h^2.
\end{align}
	Let $\Psi^{r,s}=\operatorname{Op} S^{r,s}$ be the corresponding pseudodifferential operators, which maps $H^{r,s}$ into $L^2$. Returning to the old notation where using $t$ instead of $x_1$, and assume that $P\in \mathcal C^\infty(\mathbb R, \operatorname{Op} (S(\vartheta,g)))$ be as in $(\ref{form normal})$, we get $P\in \Psi^{0,1}$.
	
	In order to prove the main theorem, we need to study the regularity of a Cauchy problem that we will state. Consider the following $N\times N$ system 
	\begin{align}
	Q=q \operatorname{Id}_N +\sum_{i=0}^{r_0-1}Q_i.
	\end{align}
	Here $q$ is a scalar operator with symbol
	\begin{align}
	q(t,x;\tau,\xi)=\tau\prod_{j=2}^{r_0}(\tau+\beta_j),
	\end{align}
	where $\beta_j\in S(\vartheta,g)$ is homogeneous and satisfies $(\ref{betaj})$. We will assume 
	\begin{align}
	Q_i=\sum_{k=0}^{i}A_{i-k}^iD_t^k,
	\end{align}
	with $A_{k'}^i\in\operatorname{Op} S(\vartheta^{k'},g)$. We are going to study the following Cauchy Problem:
	\begin{align}\label{cauchy1}
	\begin{split}
	Qu&=f\\
	D_t^k u|_{t=0}&=u_k,\ \text{for}\ k=0,...,r_0-1.
	\end{split}
	\end{align}
	We are going to assume that the $\xi\neq 0$ in $\operatorname{WF}(u)$. Hence, the restrictions are well defined.
	\begin{prop}\label{prop 7.1}
		Assume that $u\in\mathcal D'(\mathbb R^n,\mathbb C^N)$ satisfies $(\ref{cauchy1})$, and $\xi\neq 0$ in $\operatorname{WF}(u)$. If $u_k\in H^{r,s-k}$ at $(x_0,\xi_0)$ for $k=0,...,r_0-1$, $f\in H_*^{r,s-r_0+1}$ at $(t,x_0,\xi_0)$ for $0\leq t\leq t_0$, and $\xi_0'=0$, then $u\in H_*^{r,s}$ at $(t_0,x_0,\xi_0)$.
	\end{prop}
	\begin{proof}
		By conjugating with an elliptic, scalar operator with symbol in $S(\langle\xi \rangle^r \vartheta^s,g)$, we may assume that $r=s=0$. We will reduce to a first order symmetric system. Let $v_k=\lambda^{k-1} D_t^{k-1} u$ for $k=1,...,r_0$, where $\lambda\in\operatorname{Op} S(\vartheta^{-1},g)$ has symbol $\vartheta^{-1}$. Hence, $v_k=\lambda D_t v_{k-1}$ for $k=2,...,r_0$. Then $V=\ ^t(v_1,...,v_{r_0})\in\mathcal D'(\mathbb R^n,\mathbb C^{r_0N} )$, $\xi\neq 0$ in $\operatorname{WF}(V)$, and $V$ satisfies 
		\begin{align}\label{0.23}
		\begin{split}
		PV&=F\\
		V|_{t=0}&=V_0.
		\end{split}
		\end{align}
		Here $P=\operatorname{Id}_{r_0N} D_t+K$, $F=(0,...,0,\lambda^{r_0-1} f)$, $V_0=(u_0,\lambda u_1,...,\lambda^{r_0-1}u_{r_0-1})$, and $K\in\operatorname{Op} S(\vartheta,g)$ has  principal symbol $k=(k_{ij})_{1\leq i,j\leq r_0}$ such that 
		\begin{align}
		&k_{ii+1}=-\vartheta\operatorname{Id}_N\ \text{for}\ 1\leq i\leq r_0-1,\\
		&k_{r_0j}=\sum \label{kr0j} \beta_{i_1}...\beta_{i_{r_0-j+1}}\vartheta^{j-r_0}\ \text{for}\ 2\leq j\leq r_0,\\
		&k_{ij}=0 \ \text{elsewhere}
		\end{align}  
		where the sum in $(\ref{kr0j})$ is such that $2\leq i_1\leq j$ and $ i_{k-1}+1\leq i_k\leq j+k-1$ for $2\leq k\leq r_0-j+1$.
		
		We find $V_0\in H^{0,0}=L^2$ at $(x_0,\xi_0)$ and $F\in H_*^{0,0}$ at $(t,x_0,\xi_0)$ for $0\leq t\leq t_0$. Thus the result follows from the next proposition, which we will state after the following Lemma.
	\end{proof}
	\begin{lemma}\label{lemma 5.2}
		When the above assumptions are satisfied we get 
		\begin{align}
		\int_{0}^{\varepsilon}\lVert V\rVert(t) dt\leq C\varepsilon^{1/2}\bigg(\int_{0}^{\varepsilon}\lVert D_tV+KV\rVert (t) dt+\lVert V\rVert(0)\bigg)
		\end{align}
		for $V\in \mathcal C^{\infty}_0(\mathbb R^{n+1}, \mathbb C^{r_0N})$, if $\varepsilon>0$ is sufficiently small. Here $\lVert .\rVert(t)$ denotes the $L^2$ norm in the $x$ variables, depending on $t$.
	\end{lemma}
	\begin{proof}
		We are going to prove it in a similar way as the proof of Lemma $5.2$ in \cite{denckerconicalrefraction}. As $k$ is diagonalizable in $S(1,g)$, we get 
		that $k=\sum_{j=1}^{r_0}\beta_j \pi_j$, with $\beta_1=0$, 
		where $\pi_j(t,x,\xi)\in S(1,g)$ is the projection on the eigenvectors corresponding to the eigenvalues $\beta_j$ along the others when $\xi'\neq 0$. $k$ is symmetrizable with symmetrizer $M=\sum \pi_j^*\pi_j$, that is $M>0$ amd $Mk$ is symmetric. 
		Note that $k$ is symmetrizable means there exists symmetric $N\times N$ system $M(t,x,\xi)\in S(1,g)$ such that $0<c\leq M$ and $Mk-(Mk)^*\in S(1,g)$. If we put $\lVert V\rVert(t)$ to be the $L^2$ norm in the $x$-variables, depending on $t$, and we put 
		\begin{align}
		\lVert V\rVert_M^2(t)=\langle MV,V\rangle=\int \langle MV(t,x),\overline{V(t,x)}\rangle dx
		\end{align}
		then 
		\begin{align}\label{5.7}
		c\leq \frac{\lVert V\rVert_M^2(t)}{\lVert V\rVert^2(t)}\leq C
		\end{align}
		If $D_t V+KV=F$, we obtain 
		\begin{align}
		\begin{split}
		\partial_t\lVert V\rVert_M^2(t)&=\langle(\partial_t M-i(MK-K^*M))V,V\rangle(t)+\langle MF,V\rangle(t)
		+\langle MV,F\rangle(t)\\ &\leq C(\lVert V\rVert _M^2(t)+\lVert F\rVert_M^2(t)).
		\end{split}
		\end{align}
		By Gr\"{o}nwall's inequality we get, for bounded $t$,
		\begin{align}
		\lVert V\rVert_M^2(t)\leq C\bigg(\lVert V\rVert_M^2(0)+\int_{0}^{t}\lVert F\rVert_M^2(s) ds\bigg),
		\end{align}
		so $(\ref{5.7})$ and integration gives the result.
	\end{proof}
	\begin{prop}\label{regularity proposition}
		Let $P=D_t\operatorname{Id}_N+K$ where $K\in \operatorname{Op}S(\vartheta,g)$ has symbol which is diagonalizable in $S(1,g)$ with eigenvalues $0,\beta_2,...,\beta_{r_0}$ $\operatorname{mod}$ $S(1,g)$, and $\beta_j\in\mathcal C^\infty$ is homogeneous, satisfying $(\ref{betaj})$ for $j=2,...,r_0$. Assume that $V\in \mathcal D'(\mathbb R^n,\mathbb C^N)$, $\xi\neq 0$ in $\operatorname{WF}(V)$ and $V$ satisfies $(\ref{0.23})$. If $V_0\in L^2$ at $(x_0,\xi_0)$, $F\in H_*^{0,0}$ at $(t,x_0,\xi_0)$  for $0\leq  t\leq t_0$, and $\xi_0'=0$, then $V\in H_*^{0,0}$ at $(t_0,x_0,\xi_0)$. 	
		
		The condition on $k$ means that there exists a basis of eigenvectors $\{v_j\}\in S(1,g)$, with eigenvalues $0$, $\beta_2$,...,$\beta_{r_0}$ $\operatorname{mod}$ $S(1,g)$.
	\end{prop}
	The above proposition is similar to Proposition $7.2$ in \cite{denckerdoublerefraction}. To prove Proposition $7.2$ in \cite{denckerdoublerefraction}, Dencker used the parametrix constructed in \cite{denckerpropagationofsing} for $P=D_t\operatorname{Id}_N+K$, where $K\in \operatorname{Op}S(\vartheta,g)$ has principal symbol $k$ satisfying the conditions in Proposition $7.2$, with $\vartheta$, and $g$ are given as $(3.8)$, and $(3.9)$ in \cite{denckerdoublerefraction} respectively, and he used the microlocal uniqueness; see \cite{denckerpropagationofsing}. In our case; case of generalized transverse type, we are using different weight and metric, but we can still construct a parametrix for the $N_0\times N_0$ matrix, where $K\in \operatorname{Op}S(\vartheta,g)$ has principal symbol $k$ satisfying the conditions in Proposition $\ref{regularity proposition}$, and we can prove microlocal uniqueness as in \cite{denckerpropagationofsing}. The steps are similar to that in \cite{denckerpropagationofsing}, except some details are changed. Thus, we will not write about the construction of the parametrix and the microlocal uniqueness and we invite the readers to check \cite{denckerpropagationofsing}.
	\begin{proof}[Proof of Proposition \ref{regularity proposition}]
		The argument is all the same as in the proof of Proposition $7.2$ in \cite{denckerdoublerefraction}, except that to get $U\in \mathcal C^\infty(\mathbb R,\mathcal D'(\mathbb R^{n-1},\mathbb C^N))\cap L^2$, we need to prove $(5.6)$ in \cite{denckerconicalrefraction} for our case, which we proved in Lemma $\ref{lemma 5.2}$. For the parametrix and the microlocal uniqueness part we have mentioned above that it can be proven in our case in a similar way.	\end{proof}
Note that here $S_j$ are transverse at $\Sigma_2$, so their Hamiton fields are non-parallel on $\Sigma_2$. That is why we considered in the main theorem, Theorem $\ref{main theorem}$, that $M_A$ is a hypersurface near $\nu_0$, so one can consider $M_A=S_j$ for some $j$. Now, we will prove Theorem $\ref{main theorem}$ for the case of systems of generalized transverse type.
	\begin{proof}
		By multiplication and conjugation with elliptic, scalar pseudodifferential operators, we may assume  that $m=1$ and $r=0$, and using the normal form we can assume that $N=r_0$, and $P$ is of the form in Proposition $\ref{prop1 generalized}$. By using Theorem A.4 in \cite{denckerdoublerefraction} for all the terms in the expansion of $A$, we obtain that $A\in \mathcal C^\infty(\mathbb R, \Psi^0_{phg})$. As the dimension of the fiber of $\mathcal N_A\cap\mathcal N_{P}$ is $1$ at $\nu_0\in \Sigma _2$, and the dimension of the fiber of $\mathcal N_{P}$ is $r_0$ at $\Sigma_2$, we get $\operatorname{rank} \sigma(A)=r_0-1$ at $\Sigma_2$. Hence, we can conjugate by suitable elliptic systems in $\mathcal C^\infty(\mathbb R,\Psi^0_{phg})$ to get that $Au\cong (u_1,...u_{r_0-1},0)\in H^{\epsilon}$ in a conical neighborhood $U$ of $\nu_0$, for some $\epsilon>0$. Then, we find $\pi_1(\operatorname{Pol}^0(u))=\operatorname{WF}^{0}(u_{r_0})$ in $U$. By shrinking  $U$ and decreasing $\epsilon$, we may assume $Pu\in H^{\epsilon}$ in $U$. Remember that we have $P\in \Psi^{0,1}$, hence we get $Qu=\ ^tP^{\operatorname{co}}Pu\in H^{\epsilon,-r_0+1}$ there, where $\ ^tP^{\operatorname{co}}$ is the adjugate matrix of $P$. Let $Q=(q_{ij})_{i,j=1}^{r_0}$, and $P=(P_{ij})_{i,j=1}^{r_0}$. Since $q_{r_0i}$ are in $\mathcal C^\infty(\mathbb R,\operatorname{Op} S(\vartheta^{r_0-1},g))$ for $i=1,...r_0-1$, we find that $q_{r_0 r_0}u_{r_0+1}\in H^{\epsilon,-r_0+1}$ in $U$. Similarly, we find that $P_{r_0 r_0}u_{r_0}\in H^{\epsilon,-1}$, which in turns gives $D^{k-1}_tP_{r_0 r_0}u_{r_0}\in H^{\epsilon,-k}$ for $k=1,...,r_0-1$. We want to prove that $u_{r_0}\in H^{0}$ at $(t,x_0;0,\xi_0)\in U\cap\Sigma_2$ for $t<t_0$, implies $u_{r_0}\in H^{0}$ at $(t_0,x_0;0,\xi_0)=\nu_0$.
		
		Thus assume that $u_{r_0}\in H^{0}$ at $(t,x_0;0,\xi_0)\in U\cap \Sigma _2$ when $t<t_0$. We may assume that $\delta\leq 1$, then Lemma $\ref{lemma 3.9 uniaxial}$ gives that $u_{r_0}$, $P_{r_0 r_0}u_{r_0}$, $D^{k-1}_tP_{r_0 r_0}u_{r_0}$ for $k=2,...,r_0-1$  and $q_{r_0 r_0}u_{r_0}$ satisfies $(\ref{3.30})$. Then $\xi\neq 0$ in $\operatorname{WF}(u_{r_0})$, and assuming that $\delta\leq\epsilon$ in $(\ref{3.30})$ we find that $P_{r_0 r_0}u_{r_0}\in H_*^{0,-1}$, $D^{k-1}_tP_{r_0 r_0}u_{r_0}\in H_*^{0,-k}$ for $k=1,...,r_0-1$  and  $q_{r_0 r_0}u_{r_0}\in H_*^{0,-r_0+1}$ in $\pi_{0}(U\cap \Sigma_2)$, and $u_{r_0}\in H_*^{0}$ at $(t,x_0,\xi_0)\in \pi_0(U\cap \Sigma_2)$ for $t<t_0$. Since $P_{r_0 r_0}\cong D_t$ $\operatorname{mod}$ $\mathcal C^\infty (\mathbb R,\operatorname{Op} S(\vartheta,g))$, we get $P_{r_0 r_0}u_{r_0}\cong D_t u_{r_0}\in H_*^{0,-1}$ which implies, using $D^{k-1}_tP_{r_0 r_0}u_{r_0}\in H_*^{0,-k}$, that $D_t^ku_{r_0}\in H_*^{0,-k}$ for $k=1,...,r_0-1$. This gives 
		\begin{align}
		(D_t^ku_{r_0})|_{t=r}\in  H_*^{0,-k}\ \ \ \text{at} \ \ (x_0,\xi_0),\ \text{for}\ k=0,...,r_0-1,
		\end{align}
		for almost all $r<t_0$, close to $t_0$. Proposition $\ref{prop 7.1}$ (with $N=1$ and $Q=q_{r_0 r_0}$) gives $u_{r_0}\in H_*^{0}$ at $(t_0,x_0,\xi_0)$, and Lemma $\ref{lemma 3.9 uniaxial}$ gives $u_{r_0}\in H^{0}$ at $(t_0,x_0;0,\xi_0)$. 
	\end{proof}
	\section{Propagation of polarization sets for systems of MHD type}\label{section propagation of polarization sets for systems of MHD type}
	In this section, we define systems of MHD type which are also systems having their characteristic sets are union of $r_0$ hypersufaces intersecting transversally at an involutive manifold of codimension $d_0\geq 2$ and  $r_0\geq2$. However, they satisfy some assumptions different than that in case of systems of generalized transverse type. We named them systems of MHD type because we first noticed these types of systems when we considered studying the propagation of polarization sets of the linearized ideal MHD equations. Let $P\in\Psi^{m}_{phg}(X)$ be an $N\times N$ system of classical pseudodifferential operators of order $m$ on a smooth manifold $X$. Let $p=\sigma(P)$ be the principal symbol of $P$, $\det p$ the determinant of $p$ and $\Sigma=(\det p)^{-1}(0)$ the characteristic set of $P$. Assume microlocally near $\nu_0=(x_0,\xi_0)\in\Sigma$, that 
	\begin{align}\label{sigma and sigma2}
	\begin{split}
	\Sigma=&\bigcup_{j=1}^{r_0} S_j,\  r_0\geq 2, \ \text{where}\  S_j\  \text{are non-radial hypersurfaces intersecting transversally at}\\ \  &\Sigma_2=\bigcap_{j=1}^{r_0} S_j.
	\end{split}
	\end{align}
	We are interested in finding the propagation of polarization set at $\Sigma_2$, as in our application; see Section $\ref{section application}$, we know the result on $\Sigma\setminus \Sigma_2$. We assume that  
	\begin{align}\label{sigma2 manifold}
	\Sigma_2\ \text{is an involutive manifold of codimension}\ d_0\geq 2,
	\end{align}
	\begin{align}\label{detp at sigm2}
	d^j(\det p)=0\ \ \text{for}\  j\leq r_0\ \text{and}\  d^{r_0+1}(\det p)\neq 0\ \ \ \ \text{at}\ \Sigma _2.
	\end{align}
	Moreover, assume that 
	\begin{align}\label{dim fiber Np at sigma2}
	\text{the dimension of the fiber of}\ \mathcal N_P\ \text{is equal to}\ r_0+1\ \text{at}\ \Sigma_2,
	\end{align}
	\begin{align}\label{detp at Si0}
	d (\det p)= 0 \ \text{and}\ d^2 (\det p)\neq 0\ \ \text{at}\ S_{i_0}\setminus \Sigma_2,\ \text{for only one}\ i_0\in\{1,...r_0\},
	\end{align}
	and
	\begin{align}\label{detp at Sj}
	d(\det p)\neq 0 \ \text{at}\ S_j\setminus \Sigma _2\ \text{for each}\ j\in\{1,...,r_0\}, \ \text{such that} \ j\neq i_0.
	\end{align}
	Moreover, assume that $\ ^tP^{\operatorname{co}}$ the adjugate matrix of $P$ can be written as 
	\begin{align}\label{adjugate matrix}
	^tP^{\operatorname{co}}=RL_1+L_2,
	\end{align}
	with $R$ being a scalar pseudodifferential operator of order $m$, with $\sigma(R)$ vanishing on $S_{i_0}$ to the first order. $L_1$, and $L_2$ are $N\times N$ system of pseudodifferential operators of order $m(N-2)$, and $m(N-3)$ respectively. Assume also that 
	\begin{align}\label{principal symbol of L1}
	\sigma(L_1)p=f\operatorname{Id}_N,
	\end{align}
	with $\Sigma=\{f=0\}$, and $d f\neq 0$ at $S_i\setminus \Sigma_2$ for $i=\{1,...,r_0\}$.
	We are using same notation as the previous section. 
	\begin{defn}
		The system $P$ is of MHD type at $\nu_0\in\Sigma_2$, if $(\ref{sigma and sigma2})$-$(\ref{principal symbol of L1})$, and $(\ref{2.7 double refraction})$ hold microlocally near $\nu_0$.
	\end{defn}
	We want to write systems of MHD in a normal form.
	\begin{prop}\label{prop1 MHD}
		Let $P\in \Psi ^1_{phg}$ be an $N\times N$ system of MHD type at $\nu_0\in \Sigma _2$. Then by choosing suitable symplectic coordinates, we may assume that $X=\mathbb R\times \mathbb R^{n-1}$, $\nu_0=(0;(0,...,1))$, and 
		\begin{align}\label{2.3 singularities MHD}
		S_j=\{(t,x;\tau,\xi)\in T^*(\mathbb R\times \mathbb R^{n-1}): \tau+\beta_j(t,x,\xi)=0\},\ \ j=1,...,r_0,
		\end{align}
		microlocally near $\nu_0$. Here $\beta_j$ are real and homogeneous of degree $1$ in $\xi$; with
		$\beta_1\equiv 0$, satisfies in a conical neighborhood of $\nu_0$
		\begin{align}\label{betaj MHD}
		c\lvert \xi '\rvert\leq \lvert \beta_j\rvert\leq C\lvert \xi '\rvert\ \ \ j=2,...,r_0, \ \ 0<c<C,
		\end{align}
		where $(\tau,\xi',\xi'')\in\mathbb R\times\mathbb R^{d_0-1}\times \mathbb R^{n-d_0}$, which gives $\Sigma _2=\{\tau=0, \xi'=0\}$. By conjugating $P$ with elliptic, scalar Fourier integral operators, and multiplying with elliptic $N\times N$ systems of order $0$, we may assume that
		\begin{align}
		P\cong\begin{pmatrix}
		F & 0\\
		0 & E
		\end{pmatrix}\ \ \operatorname{mod}\ \mathcal C^\infty,
		\end{align}
		microlocally near $\nu_0$, where $E\in \Psi ^1_{phg}$ is an elliptic $(N-r_0-1)\times (N-r_0-1)$ system, and 
		\begin{align}\label{3.4 MHD}
		F\cong \operatorname{Id}_{r_0+1} D_t+K(t,x,D_x)\ \ \operatorname{mod}\ \mathcal C^\infty.
		\end{align}
		Here $K(t,x,D_x)\in\mathcal C^\infty(\mathbb R,\Psi ^1_{phg})$ is an $(r_0+1)\times (r_0+1)$ system, and the eigenvalues of $k(t,x,\xi)$; the principal symbol of $K(t,x,D_x)$, are $0$ (double), $\beta_2$,...$\beta_{r_0}$. 
	\end{prop}
	\begin{proof}
		In a similar way as in the proof of Proposition $\ref{prop1}$, we get $(\ref{2.3 singularities MHD})$, $(\ref{betaj MHD})$, and 
		\begin{align}
		P\cong\begin{pmatrix}
		F & 0\\
		0 & E
		\end{pmatrix}\ \ \operatorname{mod}\ \mathcal C^\infty,
		\end{align}
		microlocally near $\nu_0$, where $E\in \Psi ^1_{phg}$ is an elliptic $(N-r_0-1)\times (N-r_0-1)$ system. If $f$ is the principal symbol for $F$, then conditions $(\ref{detp at Si0})$, $(\ref{detp at Sj})$, and $(\ref{2.3 singularities MHD})$ imply 
		\begin{align}
		\det f =c\tau^2\prod_{i=2}^{r_0}(\tau+\beta_i), \ 0\neq c\in  S^{-1},
		\end{align}
		thus $\partial_{\tau}^{r_0+1}(\det f)=\det(\partial_{\tau} f)\neq 0$ at $\Sigma_2$. Same as before we get $(\ref{3.4 MHD})$.
	\end{proof}
We use the same weight and metric introduced in the previous section; see $(\ref{w})$, and $(\ref{metric generalized uniaxial})$. Thus, we have $\beta_j\in S(\vartheta,g)$.
 \begin{prop}\label{MHD vs weyl symbol}
 	Let 
 	\begin{align}\label{3.10}
 	P=\operatorname{Id}_{r_0+1} D_t+K(t,x,D_x)
 	\end{align}
 	be an $(r_0+1)\times (r_0+1)$ system with $K\in\mathcal C^\infty (\mathbb R,\Psi^1_{phg})$, such that the eigenvalues of $k=\sigma(K)$ are $0$ (double), $\beta_2$,...,$\beta_{r_0}$. Then $P$ satisfies $(\ref{2.7 double refraction})$ if and only if $k\in\mathcal C^\infty(\mathbb R, S(\vartheta,g))$.
 \end{prop}
 \begin{proof}
Same proof as in Proposition $\ref{prop1 generalized}$, just we replace $r_0$ by $r_0+1$ when needed. 	
 \end{proof}
	\begin{prop}\label{prop 3.7 mhd}
		Assume that $P$ is an $(r_0+1)\times (r_0+1)$ system of pseudodifferential operators of order $1$ on $\mathbb R^n$, on the form $(\ref{3.10})$, with $K$ is in $ \mathcal C^\infty(\mathbb R, \operatorname{Op} (S(\vartheta,g)))$  near $\nu_0\in \Sigma _2$, and $k(t,x,\xi)$ has the eigenvalues $0$ (double), $\beta_1$,..., $\beta_{r_0}$. Let $u\in\mathcal S'(\mathbb R^n,\mathbb C^{r_0+1})$ and assume $Pu\in H^{r,s}$ at $\nu_0$. Then, for every $\delta>0$ we can find $c_{\delta}$ and $C_{\delta,N}>0$ and $v_\delta\in H^{r,s+1}$ at $\nu_0$, such that $u_\delta=u-v_\delta$ satisfies
		\begin{align}\label{3.30 MHD}
		\lvert \hat {u}_\delta(\tau,\xi)\rvert \leq C_{\delta,N}\langle(\tau,\xi)\rangle^{-N}, \ \ \ \forall N,
		\end{align}
		when $\lvert\tau\rvert >c_\delta(\langle\xi \rangle^\delta+\langle \xi' \rangle)$.
	\end{prop}
	\begin{proof}
		Same proof as for Proposition $\ref{prop 3.7 generalized}$, with replacing $r_0$ by $r_0+1$ when needed. 
	\end{proof} 
	
	Now, we are ready to prove Theorem $\ref{main theorem}$ for systems of MHD type.
	\begin{proof}[Proof of Theorem $\ref{main theorem}$]
 Note that as $M_A=\pi_1(\mathcal N_A\cap\mathcal N_P\setminus 0)$ is a hypersurface near $\nu_0$, we have $M_A=S_j$ for some $j$.
		By multiplication and conjugation with elliptic, scalar pseudodifferential operators we may assume  that $m=1$ and $r=0$, and using the normal form we can assume that $N=r_0+1$, and $P$ is of the form $(\ref{3.10})$, with $k\in \mathcal C^\infty(\mathbb R, S(\vartheta,g))$. By using Theorem A.4 in \cite{denckerdoublerefraction} for all the terms in the expansion of $A$, we obtain that $A\in \mathcal C^\infty(\mathbb R, \Psi^0_{phg})$. As the dimension of the fiber of $\mathcal N_A\cap\mathcal N_{P}$ is $1$ at $\nu_0\in \Sigma _2$, and the dimension of the fiber of $\mathcal N_{P}$ is $r_0+1$ at $\Sigma_2$, we get $\operatorname{rank} \sigma(A)=r_0$ at $\Sigma_2$. Hence, we can conjugate by suitable elliptic systems in $\mathcal C^\infty(\mathbb R,\Psi^0_{phg})$ to get that $Au\cong\  ^t(u_1,...u_{r_0},0)\in H^{\epsilon}$ in a conical neighborhood $U$ of $\nu_0$, for some $\epsilon>0$. Then, we find $\pi_1(\operatorname{Pol}^0(u))=\operatorname{WF}^{0}(u_{r_0+1})$ in $U$. By shrinking  $U$ and decreasing $\epsilon$, we may assume $Pu\in H^{\epsilon}$ in $U$. Remember that we have $P\in \Psi^{0,1}$. Let $L=L_1+L_2$, where $L_1$ and $L_2$ are given in $(\ref{adjugate matrix})$, and let $Q=LP$. We have  $Qu=LPu\in H^{\epsilon,-r_0+1}$ there. Let $Q=(q_{ij})_{i,j=1}^{r_0+1}$, and $P=(P_{ij})_{i,j=1}^{r_0+1}$. By $(\ref{adjugate matrix})$ and $(\ref{principal symbol of L1})$ we have $q_{ij}$ are in $\mathcal C^\infty(\mathbb R,\operatorname{Op} S(\vartheta^{r_0-1},g))$ for $i\neq j$, we find that $q_{r_0+1 r_0+1}u_{r_0+1}\in H^{\epsilon,-r_0+1}$ in $U$. Similarly, we find that $P_{r_0+1 r_0+1}u_{r_0+1}\in H^{\epsilon,-1}$, which in turns gives $D^{k-1}_tP_{r_0+1 r_0+1}u_{r_0+1}\in H^{\epsilon,-k}$ for $k=1,...,r_0-1$. We want to prove that $u_{r_0+1}\in H^{0}$ at $(t,x_0;0,\xi_0)\in U\cap\Sigma_2$ for $t<t_0$, implies $u_{r_0+1}\in H^{0}$ at $(t_0,x_0;0,\xi_0)=\nu_0$.
		
		Thus assume that $u_{r_0+1}\in H^{0}$ at $(t,x_0;0,\xi_0)\in U\cap \Sigma _2$ when $t<t_0$. We may assume that $\delta\leq 1$, then Lemma $\ref{lemma 3.9 uniaxial}$ gives that $u_{r_0+1}$, $P_{r_0+1 r_0+1}u_{r_0+1}$, $D^{k-1}_tP_{r_0+1 r_0+1}u_{r_0+1}$ for $k=2,...,r_0-1$  and $q_{r_0+1 r_0+1}u_{r_0+1}$ satisfies $(\ref{3.30 MHD})$. Then $\xi\neq 0$ in $\operatorname{WF}(u_{r_0+1})$, and assuming that $\delta\leq\epsilon$ in $(\ref{3.30})$ we find that $P_{r_0+1 r_0+1}u_{r_0+1}\in H_*^{0,-1}$, $D^{k-1}_tP_{r_0+1 r_0+1}u_{r_0+1}\in H_*^{0,-k}$ for $k=1,...,r_0-1$  and  $q_{r_0+1 r_0+1}u_{r_0+1}\in H_*^{0,-r_0+1}$ in $\pi_{0}(U\cap \Sigma_2)$, and $u_{r_0+1}\in H_*^{0}$ at $(t,x_0,\xi_0)\in \pi_0(U\cap \Sigma _2)$ for $t<t_0$. Since $P_{r_0+1 r_0+1}\cong D_t$ $\operatorname{mod}$ $\mathcal C^\infty (\mathbb R,\operatorname{Op} S(\vartheta,g))$, we get $P_{r_0+1 r_0+1}u_{r_0+1}\cong D_t u_{r_0+1}\in H_*^{0,-1}$ which implies using $D^{k-1}_tP_{r_0+1 r_0+1}u_{r_0+1}\in H_*^{0,-k}$ that $D_t^ku_{r_0+1}\in H_*^{0,-k}$ for $k=1,...,r_0-1$. This gives 
		\begin{align}
		(D_t^ku_{r_0+1})|_{t=r}\in  H_*^{0,-k}\ \ \ \text{at} \ \ (x_0,\xi_0),\ \text{for}\ k=0,...,r_0-1
		\end{align}
		for almost all $r<t_0$, close to $t_0$. Proposition $\ref{prop 7.1}$ (with $N=1$ and $Q=q_{r_0+1 r_0+1}$) gives $u_{r_0+1}\in H_*^{0}$ at $(t_0,x_0,\xi_0)$, and Lemma $\ref{lemma 3.9 uniaxial}$ gives $u_{r_0+1}\in H^{0}$ at $(t_0,x_0;0,\xi_0)$. 
	\end{proof}
	\section{Application}\label{section application}
	Magnetohydrodynamics, or MHD couples Maxwell's equations with hydrodynamics to describe the behavior of electrically conducting fluids under the influence of electromagnetic fields. In this section, we want to consider the simplest form of MHD, which is the Ideal MHD to study the propagation of polarization set of the solution of the linearized ideal MHD equations. Ideal MHD, assumes that the fluid has so little resistivity that it can be treated as a perfect conductor. See \cite{lecturesmhd} to know more about MHD equations.
	
	We will show, under some assumptions, that the linearized ideal MHD equations is of real principal type. As we mentioned before, systems of real principal type were defined by Dencker; see \cite{denckerprincipaltype}, who studied the propagation of their solutions, and showed that the propagation of polarization sets is governed by a certain connection on sections of the kernel subbundle, $\ker p_2$ where $p_2$ is the principal symbol of the system. In \cite{hansenlagrangiansolutions}, Hansen and R\"{o}hrig merged the theory of real principal type systems with the calculus of Fourier integral operators and constructed a Fourier integral solution for system of real principal type, and derived a transport equation for the principal symbol of this solution (note that disregarding half densities this transport equation is the connection introduced by Dencker).  
	
	We will divide this section into two subsections: first, we introduce the ideal MHD equations and its linearization. In Section $\ref{section application transport eq}$, we write the linearized ideal MHD equations in the form of a wave equation $P\beta=0$ where $\beta$ is the displacement vector and $P$ is a second order  $3\times3$ system; see \cite[Lecture 20]{lecturesmhd}, and we show that under some assumptions, the characteristic variety of $P$ is disjoint union of the Shear Alfv\'{e}n wave, the slow magnetosonic wave and the fast magnetosonic wave; see \cite[Lecture 24]{lecturesmhd}. Moreover, we show that, under the considered assumptions, $P$ is of real principal type and we calculate the transport equation on $\operatorname{Char}P$, as an application to the result in \cite{hansenlagrangiansolutions}. In Section $\ref{section application MHD type system}$, we return to the linearized ideal MHD equations, and we study the propagation of polarization sets. It turns out that we can consider different cases, some in which we have our system is of real principal type, some in which our system is of uniaxial type, and we have a case where our system is of MHD type.
	
	The set of equations describing the ideal MHD are 
	\begin{align}\label{1}
	\begin{cases}
	\partial_t\rho +u\cdot\nabla \rho +\rho\operatorname{div}u=0,\\
	\rho(\partial_t u+u\cdot\nabla u)+\nabla p+ H\times\operatorname{curl} H=0,
	\\\partial_tH-\nabla\times(u\times H)=0,\\
	\partial_tp+u\cdot\nabla p+\gamma p\operatorname{div} u=0,
	\end{cases}
	\end{align}
	where $\rho$, $p\in \mathbb R$ denotes the density and the pressure respectively. $u\in\mathbb R^3$ is the fluid velocity, $H\in\mathbb R^3$ is the magnetic field, and $\gamma $ is the adiabatic index, see \cite[Lecture \ 20]{lecturesmhd}.
	\\Assuming a stationary equilibrium the linearized equations of $(\ref{1})$ about $(\rho,H,p)$ is:
	\begin{subequations}
		\begin{align}
		&\partial_t\dot{\rho} =-\rho\operatorname{div}\dot{u}-\dot{u}\cdot\nabla\rho,\label{2a}\\
		&\rho\partial_t \dot{u}=-\nabla\dot{p}-H\times\operatorname{curl}\dot{H}-\dot{H}\times\operatorname{curl}H,\label{2b}
		\\&\partial_t\dot{H}=\nabla\times(\dot{u}\times H),\label{2c}\\
		&\partial_t\dot{p}=-\gamma p\operatorname{div}\dot{u}-\dot{u}\cdot\nabla p,\label{2d} 
		\end{align}
		\label{2}
	\end{subequations}
	where $(\rho,H,p)$ are the values in the equilibrium state (that is the solutions of the Ideal MHD equations when $\partial/\partial t=0$, and as we assumed stationary equilibrium we have $u=0$). Note that we used that 
	\begin{align}\label{condition equilibrium}
	\nabla p+H\times \operatorname{curl}H=0,
	\end{align}
	which we get from the stationary equilibrium assumption.
	\subsection{The ideal MHD wave equation and the transport equation}\label{section application transport eq}
	We can write the linearized ideal MHD equations (\ref{2}) in the form of a wave equation; see \cite[Lecture \ 20]{lecturesmhd}. Consider the displacement $\beta$ to be defined by 
	\begin{align}\label{4}
	\dot{u}=\frac{\partial \beta}{\partial t}.
	\end{align}
	Substituting (\ref{4}) in (\ref{2a}), (\ref{2c}), and (\ref{2d}) respectively, and integrating with respect to $t$ we get
	\begin{subequations}
		\begin{align}
		&\dot{\rho}=-\rho \operatorname{div} \beta-\beta\cdot\nabla \rho,\label{5a}
		\\
		& \dot{H}=\nabla \times(\beta\times H).\label{5b}\\
		& \dot{p}=-\gamma p \operatorname{div} \beta-\beta\cdot\nabla p,\label{5c}
		\end{align}
		\label{5}
	\end{subequations}
	Replace now (\ref{5b}), (\ref{5c}), and $(\ref{4})$ in (\ref{2b}) to get
	\begin{align}
	\rho\frac{\partial^2\beta}{\partial t^2}=&\gamma\nabla( p \operatorname{div} \beta)+\nabla(\beta\cdot\nabla p)+(\nabla\times(\nabla \times(\beta\times H)))\times H+(\nabla\times H)\times(\nabla\times(\beta\times H)).\label{6}
	\end{align}
	Equation (\ref{6}) is the ideal MHD wave equation.
	Consider from now on $P$ where 
	\begin{align}
	\begin{split}
	P\beta &=-\rho\frac{\partial^2\beta}{\partial t^2}+\gamma\nabla( p \operatorname{div} \beta)+\nabla(\beta\cdot\nabla p)+(\nabla\times(\nabla \times(\beta\times H)))\times H\\&+(\nabla\times H)\times(\nabla\times(\beta\times H))\\
	&=0.
	\end{split}
	\end{align}
	Now, we want to give the Characteristic variety of $P$ under some assumptions; for this part we refer to \cite[Lectures\ 23 and 24]{lecturesmhd}.
	\begin{lemma}\label{lemma char disjoint} Assume $c^2=\gamma p/\rho>0$, $0<\lvert H\rvert ^2\neq \rho c^2$, $\xi\cdot H\neq 0$, and $\xi\times H\neq 0$ . The characteristic variety of $P$ is disjoint union of the Shear Alfv\'{e}n wave, the slow magnetosonic wave, and the fast magnetosonic wave characteristic varieties $\{q_1=0\}$, $\{q_2=0\}$, and  $\{q_3=0\}$ respectively, where 
		\begin{align}
		\begin{cases}
		&q_1=\rho \tau^2-(H.\xi)^2,
		\\&q_2=\rho(\tau^2-c_s^2(x,\xi))
		\\&q_3=\rho(\tau^2-c_f^2(x,\xi)),
		\end{cases}
		\end{align}
		with
		\begin{align*}
		c_f^2(x,\xi):=\frac{1}{2}\big((c^2+h^2)\xi^2+\sqrt{(c^2-h^2)^2\xi^4+4b^2c^2\xi^2}\big),
		\end{align*}
		\begin{align*}
		c_s^2(x,\xi):=\frac{1}{2}\big((c^2+h^2)\xi^2-\sqrt{(c^2-h^2)^2\xi^4+4b^2c^2\xi^2}\big),
		\end{align*}
		where $c^2=\gamma p/\rho>0$, and $h^2=\lvert H\rvert^2/\rho$ are considered to be the sound speed and the Alfv\'{e}n speed respectively, and $b^2=\lvert \xi\times H\rvert^2/\rho.$\end{lemma} 
	\begin{proof}	We have
		\begin{align}
		p_2\beta=\rho\tau^2 \beta-\gamma p\xi(\xi\cdot \beta)-(\xi\times(\xi\times(\beta\times H)))\times H,\label{9}
		\end{align}
		with $p_2$ is the principal symbol of $P$.
		Considering $v=H/{\sqrt{\rho}}$, equation (\ref{9}) can be written as 
		\begin{align}
		\rho(\tau^2- (\xi\cdot v)^2)\beta=\rho((c^2+h^2)(\xi\cdot\beta)-(v\cdot\beta)(\xi\cdot v))\xi-\rho v(\xi\cdot\beta)(\xi\cdot v).\label{10}
		\end{align}
		Without loss of generality, let $v=\lvert H\rvert/\sqrt{\rho} \ \hat{e}_z$, $\xi=\xi_{\perp}\hat{e}_x+\xi_{\parallel}\hat{e}_z$ with $\xi^2=\xi_{\perp}^2+\xi_{\parallel}^2$, and $\beta=\beta_x\hat{e}_x+\beta_y\hat{e}_y+\beta_z\hat{e}_z$ with $\hat{e}_x$, $\hat{e}_y$, and $\hat{e}_z$ are unit vectors that points in the direction of the x-axis, y-axis, and z-axis respectively. Substituting this in equation (\ref{10}), we find 
		\begin{subequations}
			\begin{align}
			&\text{x-component:}\ \  \rho\tau^2\beta_x=\rho c^2\xi_{\parallel}\xi_{\perp}\beta_z+\rho (h^2\xi^2+c^2\xi_{\perp}^2)\beta_x,\label{11a}
			\\&\text{y-component:}\ \ 
			\rho\tau^2\beta_y=\rho h^2\xi_{\parallel}^2\beta_y,\label{11b}
			\\& \text{z-component:}\ \ \rho\tau^2\beta_z=\rho(c^2\xi_{\parallel}\xi_{\perp})\beta_x+\rho c^2\xi_{\parallel}^2\beta_z.\label{11c}	
			\end{align}
		\end{subequations}
		Notice that the y-component decouples from the x- and z-components. This immediately gives 
		\begin{align}
		\rho\tau^2=\rho h^2\xi_{\parallel}^2,
		\end{align}
		This is the shear Alfv\'{e}n wave.
		The characteristic equation for the coupled x-and z-component is 
		\begin{align}\label{8}
		\rho^2\tau^4-\rho^2\tau^2\xi^2(c^2+h^2)+\rho^2 c^2h^2\xi_{\parallel}^2\xi^2=0.
		\end{align}
		Hence, we get 
		\begin{align}
		\rho\tau^2=\frac{\rho}{2}\bigg((c^2+h^2)\xi^2\pm\sqrt{(c^2-h^2)\xi^4+4c^2h^2\xi_{\perp}^2\xi^2}\bigg).
		\end{align}
		Still we want to prove that $\{q_1=0\}$, $\{q_2=0\}$, and $\{q_3=0\}$ are disjoint. Dividing $(\ref{8})$ by $\rho^2$, it can be written as
		\begin{align}
		(\tau^2- h^2 \xi_{\parallel}^2)( \tau^2- c^2\xi^2)-\tau^2h^2\xi_{\perp}^2.
		\end{align}
		Consider $R(X)=(X^2- h^2 \xi_{\parallel}^2)(X^2- c^2\xi^2)-X^2  h^2\xi_{\perp}^2
		$, $\{R\leq 0\}=[c_s^2,c_f^2]$ and $R(X)\leq 0$ for $X\in[\min (h^2\xi_{\parallel}^2,c^2\xi^2),$\\$\max(h^2\xi_{\parallel}^2,c^2\xi^2)]$. Thus, 
		\begin{subequations}
			\begin{align}
			&c_f^2\geq\max(h^2\xi_{\parallel}^2,c^2\xi^2)\geq h^2\xi_{\parallel}^2,\\
			&c_s^2\leq\min(h^2\xi_{\parallel}^2,c^2\xi^2)\leq h^2\xi_{\parallel}^2.
			\end{align}
			As $h^2\xi_{\parallel}^2\neq 0$, we have $R(h^2\xi_{\parallel}^2)=-h^2\xi_{\parallel}^2c^2\xi^2\leq 0$. Hence, $c_s^2< h^2\xi_{\parallel}^2< c_f^2$.
		\end{subequations} 
	\end{proof}
	Suppose that the conditions of Lemma $\ref{lemma char disjoint}$ are satisfied. Now, we are interested in calculating the transport equation as in \cite{hansenlagrangiansolutions}. 
	The full symbol of $P$ is $p_2+p_1+p_0$, with $p_2$ is the principal symbol of $P$ homogeneous of degree $2$, and $p_1$ and $p_0$ are homogeneous terms of degree $1$ and $0$ respectively. One can check that the principal symbol of $P$ is   
	\begin{align}\label{principal symbol mathematika}
	p_2=(\rho\tau^2-(H.\xi)^2)\operatorname{Id}_3-(\gamma p+\lvert H\rvert^2)\xi\otimes\xi+(H.\xi)(\xi\otimes H+H\otimes \xi),
	\end{align}
	and 
	\begin{align*}
	p_1=&i\gamma (\nabla p)\otimes \xi+i\xi\otimes(\nabla p)+i(\nabla (H\cdot \xi)\cdot H+(H\cdot\xi)\operatorname {div} H)\operatorname{Id}_3+\frac{i}{2}(\xi\otimes\nabla \lvert H\rvert^2)+i(\nabla \lvert H\rvert^2\otimes \xi)\\
	&-2i(H\cdot\nabla)(H\otimes\xi)-i(\nabla(H\cdot\xi))\otimes H-i(H\cdot\xi)(\nabla\otimes H)-i(\nabla\cdot H)(\xi\otimes H).
	\end{align*} 
	Using $(\ref{condition equilibrium})$, we get 
	\begin{align}\label{p1 mathematika}
	\begin{split}
	p_1&=i\gamma (\nabla p)\otimes \xi+i\xi\otimes(\nabla p)-i(\nabla p)\otimes\xi+i(\nabla (H\cdot \xi)\cdot H+(H\cdot\xi)\operatorname {div} H)\operatorname{Id}_3+\frac{i}{2}(\xi\otimes\nabla \lvert H\rvert^2)\\
	&+\frac{i}{2}(\nabla \lvert H\rvert^2\otimes \xi)-i(H\cdot\nabla)(H\otimes\xi)-i(\nabla(H\cdot\xi))\otimes H-i(H\cdot\xi)(\nabla\otimes H)-i(\nabla\cdot H)(\xi\otimes H).
	\end{split}
	\end{align}
	One can check $(\ref{principal symbol mathematika})$, $(\ref{p1 mathematika})$, and the calculations given below by using "mathematica" for example.
	
	Let $\Gamma_1$, $\Gamma_2$, and $\Gamma_3$ be disjoint conic neighborhoods of $\{q_1=0\}$, $\{q_2=0\}$, and $\{q_3=0\}$ respectively. Set $q=q_1$ in $\Gamma_1$, $q=q_2$ in $\Gamma_2$, and $q=q_3$ in $\Gamma_3$.
	\begin{prop}
		$P$ is of real principal type with respect to the Hamilton field $H_q$ of $\operatorname{Char} P=\{q=0\}$.
	\end{prop}
	\begin{proof}
		We have $q_1$, $q_2$, and $q_3$ are scalar real principal type.  Let
		\begin{align}
		w^2:= \lvert H\rvert^2\xi\otimes\xi+\lvert \xi\rvert^2H\otimes H-(H\cdot \xi)(H\otimes \xi+\xi\otimes H)
		\end{align}
		In $\Gamma_1$  we take 
		\begin{align}
		\tilde{p}_2=\operatorname{Id}_3+\frac{q_1}{q_2 q_3}(\gamma p+\lvert H\rvert^2)\xi\otimes \xi-\frac{q_1}{q_2 q_3}(&(H\cdot\xi)(\xi\otimes H+H\otimes\xi))+\frac{(H\cdot \xi)^2 w^2}{q_2 q_3}
		\end{align}
		to get $\tilde{p}_2p_2=q_1\operatorname{Id}_3$.
		In $\Gamma_2$ we take 
		\begin{align}
		\tilde{p}_2=\frac{q_2}{q_1}\operatorname{Id}_3+\frac{1}{ q_3}(\gamma p+\lvert H\rvert^2)\xi\otimes \xi-\frac{1}{ q_3}(&(H\cdot\xi)(\xi\otimes H+H\otimes\xi))+\frac{(H\cdot \xi)^2 w^2}{q_1 q_3}
		\end{align}
		to get $\tilde{p}_2p_2=q_2\operatorname{Id}_3$. In $\Gamma_3$ we take 
		\begin{align}
		\tilde{p}_2=\frac{q_3}{q_1}\operatorname{Id}_3+\frac{1}{ q_2}(\gamma p+\lvert H\rvert^2)\xi\otimes \xi-\frac{1}{ q_2}(&(H\cdot\xi)(\xi\otimes H+H\otimes\xi))+\frac{(H\cdot \xi)^2 w^2}{q_1 q_2}
		\end{align}
		to get $\tilde{p}_2p_2=q_3\operatorname{Id}_3$.\end{proof} 
	\begin{remark}
		$\bullet$ The principal symbol of $P$ calculated before 
		can be written as 
		\begin{align}
		p_2=q_1 \pi_1+q_2\pi_2+q_3\pi_3,
		\end{align}
		where\begin{align}
		\pi_1=\operatorname{Id}_3+\frac{ w^2}{(H\cdot \xi)^2-\lvert H\rvert^2\lvert\xi\rvert^2},
		\end{align}
		\begin{align}
		\pi_2=\frac{1}{ \rho c_s^2(x,\xi)-\rho c_f^2(x,\xi)}\bigg((\gamma p+\lvert H\rvert^2)\xi\otimes \xi-(H\cdot\xi)(\xi\otimes H+H\otimes\xi)+\frac{(H\cdot \xi)^2 w^2}{\rho c_s^2(x,\xi)-(H\cdot \xi)^2}\bigg),
		\end{align}
		and
		\begin{align}
		\pi_3=\frac{1}{ \rho c_f^2(x,\xi)-\rho c_s^2(x,\xi)}\bigg((\gamma p+\lvert H\rvert^2)\xi\otimes \xi-(H\cdot\xi)(\xi\otimes H+H\otimes\xi)+\frac{(H\cdot \xi)^2 w^2}{\rho c_f^2(x,\xi)-(H\cdot \xi)^2}\bigg),
		\end{align}
		with $\pi_1,\pi_2$ and $\pi_3$ are orthogonal projectors and $\pi_1+\pi_2+\pi_3=\operatorname{Id}_3$\\
		$\bullet$ In $\Gamma_1$,
		\begin{align}
		\tilde{p}_2=\pi_1+\frac{q_1}{q_2}\pi_2+\frac{q_1}{q_3}\pi_3,
		\end{align}
		and set $\pi=\pi_1$ .
		\\In $\Gamma_2$, 
		\begin{align}
		\tilde{p}_2=\pi_2+\frac{q_2}{q_1}\pi_1+\frac{q_2}{q_3}\pi_3,
		\end{align}
		and set $\pi=\pi_2$.\\
		In $\Gamma_3$,
		\begin{align}
		\tilde{p}_2=\pi_3+\frac{q_3}{q_1}\pi_1+\frac{q_3}{q_2}\pi_2,
		\end{align}
		and set $\pi=\pi_3$.\\
		On $\operatorname{Char}P$, $\tilde{p}_2=\pi$, $p_2\pi=0=\pi p_2$, and $p_2a=0$ if and only if $a=\pi a$.
	\end{remark}
	In what follows, let $X=\mathbb R\times\mathbb R^3$, $\Lambda\subset T^*X\setminus 0$ be a closed Lagrangian submanifold of the characteristic set of $P$, and let $\Omega_{\Lambda}^{1/2}$ denote the half-density bundle of $\Lambda$. $S^{\mu+1}(\Lambda, (\Omega_\Lambda^{1/2})^3)$ is the space of symbols of the space of Lagrangian distributions $I^\mu(X,\Lambda;(\Omega_X^{1/2})^3)$; see \cite[Section 25.1]{hormander4} to read more about Lagrangian distributions.\\
	From \cite{hansenlagrangiansolutions}, we know that there is a first order differential operator $\mathcal T_{P,H_q}$ on $\Lambda$, uniquely determined by $P$ and $H_q$ which maps $a$ a  $3$-vector of half densities with $p_2 a=0$ to $3$-vector of half densities where 
	\begin{align}
	\mathcal T_{P,H_q} a=\mathcal L_{H_q}a+\frac{1}{2}\{\tilde{p}_2,p_2\}a+i\tilde{p}_2p^sa.
	\end{align}
 Here
 \begin{align}\label{subprincipal symbol}
p^s=p_1-\frac{1}{2i}\sum_j \frac{\partial^2 p_2}{\partial x_j\partial\xi_j},
\end{align}
is the subprincipal symbol of $P$, $\mathcal  L_{H_q}$ is the Lie derivative with respect to $H_q$, and $\{.\}$ is the Poisson bracket.
	\begin{lemma}
		$\tilde{p}_2 p^s \pi=0$ on $\operatorname{Char}P$.
	\end{lemma}
	\begin{proof}
		Differentiating $p_2$ we get
		\begin{align*}
		\sum_{j=1}^{3}\frac{\partial^2 p_2}{\partial x_j\partial\xi_j}=&-2(\nabla(H\cdot\xi)\cdot H+(H\cdot\xi)\operatorname{div}H)\operatorname{Id}_3-(\gamma\nabla p+\nabla(\lvert H\rvert^2))\otimes\xi-\xi\otimes(\gamma \nabla p+\nabla \lvert H\rvert^2)
		\\&+(\operatorname{div} H+ H\cdot\nabla)(\xi\otimes H+H\otimes\xi)+(\nabla(H\cdot \xi))\otimes H+(H\cdot \xi)(\nabla\otimes H+(\nabla\otimes H)^\intercal)\\&+H\otimes(\nabla(H\cdot\xi))	
		\end{align*}
		Therefore, the subprincipal is given as follows:
		\begin{align*}
		2ip^s=&\gamma(\xi\otimes\nabla p-\nabla p\otimes\xi)+\operatorname{div}H(\xi\otimes H-H\otimes\xi)+(H\cdot\xi)(-(\nabla\otimes H)^\intercal+(\nabla\otimes H))\\&+(H\cdot\nabla)(H\otimes\xi-\xi\otimes H)+((\nabla(H\cdot\xi))\otimes H-H\otimes\nabla(H\cdot\xi))+2(\nabla p\otimes\xi-\xi\otimes\nabla p).
		\end{align*}
		We have
		\begin{align}
		2i\tilde{p}_2 p^s\pi=2i\pi p^s\pi\ \ \ \ \text{on} \operatorname{Char}P.
		\end{align}
		Using that $2ip^s$ is a $3\times 3$ skew-symmetric matrix with zero entries on the diagonal, and $\pi$ is a symmetric matrix we get that $2i\pi p^s\pi$ is a $3\times 3$ skew-symmetric matrix. Therefore to prove that it vanishes, it suffices to show that its rank is $<2$. Since $\pi$ is projection we have rank $\pi$ =trace $\pi$. Calculating the trace of $\pi$ we get that rank $\pi$=1 and hence we proved the lemma.
	\end{proof}
	\begin{lemma}
		Let $\Lambda\subset \operatorname{Char}P$ be a conic Lagrangian submanifold. Let $a\in S^{\mu+1}(\Lambda,(\Omega_\Lambda^{1/2})^3)$. Then, $\{\tilde{p}_2,p_2\}\pi a=-2(H\pi)a$ on $\operatorname{Char}P$.
	\end{lemma}
	\begin{proof} We will prove the result for $\pi=\pi_1$ and the same argument applies for $\pi_2$ and $\pi_3$.
		We have in a conic neighborhood of $\{q_1=0\}$, $p_2=q_1\pi_1+q_2\pi_2+q_3\pi_3$, and $\tilde{p}_2=\pi_1+\frac{q_1}{q_2}\pi_2+\frac{q_1}{q_3}\pi_3$. Using that $\pi^2=\pi$, and $\{q_1,q_1\}=0$ we get
		\begin{align*}
		\{\tilde{p}_2,p_2\}\pi_1 a=&\{\pi_1,q_1\}\pi_1 a+q_2\{\pi_1,\pi_2\}\pi_1 a+q_3\{\pi_1,\pi_3\}\pi_1 a+\pi_2\{q_1,\pi_2\}\pi_1 a+\frac{q_3}{q_2}\pi_2\{q_1,\pi_3\}\pi_1 a\\&+\frac{q_2}{q_3}\pi_3\{q_1,\pi_2\}\pi_1 a+\pi_3\{q_1,\pi_3\}\pi_1 a.
		\end{align*}
		Using that $H\pi_j=H\pi_j^2=(H\pi_j)\pi_j+\pi_j(H\pi_j)$ for $j=2,3$, we get $\pi_j\{q_1,\pi_j\}\pi_1 a=\{q_1,\pi_j\}\pi_1 a$ for $j=2,3$, and using $0=H(\pi_2\pi_3)=\pi_2(H\pi_3)+(H\pi_2)\pi_3$ and $0=H(\pi_3\pi_2)=\pi_3(H\pi_2)+(H\pi_3)\pi_2$ we get $\pi_2\{q_1,\pi_3\}\pi_1 a=\pi_3\{q_1,\pi_2\}\pi_1 a=0$.
		\\Using $\pi_2=\operatorname{Id}_3-\pi_1-\pi_3$, $\{\pi_1,\operatorname{Id}_3\}=\{q_1,\operatorname{Id}_3\}=0$ and $\{\pi_1,\pi_1\}=0$ we get
		\begin{align*}
		\{\tilde{p}_2,p_2\}\pi_1 a=-2(H\pi_1)\pi_1 a+(q_3-q_2)\{\pi_1,\pi_3\}\pi_1 a.
		\end{align*}
		Now, we want to prove that $\{\pi_1,\pi_3\}\pi_1a=0$. We have $\partial_{\xi_i}\pi_1=\partial_{\xi_i}\pi_1^2=\pi_1\partial_{\xi_i}\pi_1+\partial_{\xi_i}\pi_1 \pi_1$, and similarly $\partial_{x_i}\pi_1=\partial_{x_i}\pi_1^2=\pi_1\partial_{x_i}\pi_1+\partial_{x_i}\pi_1 \pi_1$. Also, $0=\partial_{\xi_i}(\pi_1\pi_3)=\pi_1\partial_{\xi_i}\pi_3+\partial_{\xi_i}\pi_1\pi_3$. Hence, $\pi_1\partial_{\xi_i}\pi_3=-\partial_{\xi_i}\pi_1\pi_3$. Similarly, we have $\pi_1\partial_{x_i}\pi_3=-\partial_{x_i}\pi_1 \pi_3$. Combining these together we get 
		\begin{align*}
		\{\pi_1,\pi_3\}\pi_1 a=\pi_1\{\pi_1,\pi_3\}\pi_1 a.
		\end{align*}
		We have 
		\begin{align*}
		\pi_1\{\pi_1,\pi_3\}\pi_1 a=\sum_{i=1}^{3}\pi_1 (\partial_{\xi_i}\pi_1\partial_{x_i}\pi_3-\partial_{x_i}\pi_1\partial_{\xi_i}\pi_3)\pi_1 a.
		\end{align*}
		Moreover,
		\begin{align*}
		\pi_1\partial_{\xi_i}\pi_1\partial_{x_i}\pi_3\pi_1 a&=\pi_1\partial_{\xi_i}\pi_1(\pi_3\partial_x\pi_3+\partial_{x_i}\pi_3\pi_3)\pi_1 a=\pi_1(\partial_{\xi_i}\pi_1\pi_3)\partial_x\pi_3\pi_1 a=-\pi_1\partial_{\xi_i}\pi_3(\partial_{x_i}\pi_3\pi_1) a\\&=\pi_1(\partial_{\xi_i}\pi_3\pi_3)\partial_{x_i}\pi_1 a=\pi_1(\partial_{\xi_i}\pi_3-\pi_3\partial_{\xi_i}\pi_3)\partial_{x_i}\pi_1 a=\pi_1\partial_{\xi_i}\pi_3\partial_{x_i}\pi_1 a.
		\end{align*}
		Using that $\pi_1 a=a$ on $\operatorname{Char}P$, we get $\pi_1\partial_{\xi_i}\pi_1\partial_{x_i}\pi_3\pi_1 a=\pi_1\partial_{\xi_i}\pi_3\partial_{x_i}\pi_1 \pi_1 a$. Therefore, 
		\begin{align*}
		\pi_1\{\pi_1,\pi_3\}\pi_1 a=\sum_{i=1}^{3}\pi_1 (\partial_{\xi_i}\pi_3\partial_{x_i}\pi_1-\partial_{x_i}\pi_1\partial_{\xi_i}\pi_3)\pi_1 a.
		\end{align*}
		We have $\partial_{\xi_i}\pi_3\partial_{x_i}\pi_1-\partial_{x_i}\pi_1\partial_{\xi_i}\pi_3$ is a $3\times3$ skew-symmetric matrix with the entries in the diagonal equal to zero. So same as before we get $\pi_1\{\pi_1,\pi_3\}\pi_1 a=0$ as the rank of $\pi_1$ equal to $1$. Hence, the lemma is proved.
	\end{proof}
	\begin{prop}
		Let $\Lambda\subset\operatorname{Char}P$ be a conic Lagrangian submanifold. Let $a\in S^{\mu+1}(\Lambda,(\Omega_\Lambda^{1/2})^3)$ with $p_2 a=0$. Then
		\begin{align}
		\mathcal T_{P,H}a=\mathcal L_H a-(H\pi)a \ \ \ \ on \ \ \Lambda.
		\end{align}
	\end{prop}
	\subsection{Propagation of polarization sets for the linearized ideal MHD equations}\label{section application MHD type system}
	Note that $(\ref{2})$ is hyperbolic symmetric with symmetrizer 
	\begin{align}
	S=\begin{pmatrix}
	\gamma p&0&0&0&0&0&0&-\rho\\
	0&\rho&0&0&0&0&0&0\\
	0&0&\rho&0&0&0&0&0\\
	0&0&0&\rho&0&0&0&0\\
	0&0&0&0&1&0&0&0\\
	0&0&0&0&0&1&0&0\\
	0&0&0&0&0&0&1&0\\
	-\rho&0&0&0&0&0&0&\frac{1+\rho^2}{\gamma p}
	\end{pmatrix}.
	\end{align} The principal symbol of $(\ref{2})$ is 
	\begin{align}\label{4'}
	\begin{cases}
	\tau\dot{\rho} +\rho(\xi\cdot\dot{u})=0,\\
	\tau \dot{u}+\rho^{-1}\dot{p}\xi+\rho^{-1} H\times(\xi\times\dot{H})=0,
	\\\tau\dot{H}+(\xi\cdot\dot{u})H-(H\cdot\xi) \dot{u}=0,\\
	\tau\dot{p}+\gamma p(\xi\cdot\dot{u})=0.
	\end{cases}
	\end{align}
	We use here the notation $\xi=(\xi_1,\xi_2,\xi_3)$ for the spatial frequencies and $\xi=\lvert\xi\rvert\hat{\xi}$, $\dot{u}_{\parallel}=\hat{\xi}\cdot\dot{u}$, $\dot{u}_{\perp}=\dot{u}-\dot{u}_{\parallel}\hat{\xi}=-\hat{\xi}\times (\hat{\xi}\times \dot{u})$.
	\\ We write $(\ref{4'})$ in the general form $\tau\dot{U}+A(U,\xi)\dot{U}=0$ with parameters $U=(\rho,H,p)$, and $\dot{U}=(\dot{\rho},\dot{u},\dot{H},\dot{p})$.
	
	We have the following result
	\begin{lemma}\label{12}
		Assume that $c^2=\gamma p/\rho>0$. The eigenvalues of $A(U,\xi)$ are 
		\begin{align}
		\begin{cases}
		\lambda_0=\lambda_4=0,\\
		\lambda_{\pm 1}=\pm c_s(\hat{\xi})\lvert\xi\rvert,\\
		\lambda_{\pm 2}=\pm (\xi\cdot H)/\sqrt{\rho},\\
		\lambda_{\pm 3}=\pm c_f(\hat{\xi})\lvert\xi\rvert,\\
		\end{cases}
		\end{align}
		with $\hat{\xi}=\xi/\lvert\xi\rvert$ and 
		\begin{align}\label{cf hat}
		c_f^2(\hat{\xi}):=\frac{1}{2}\big((c^2+h^2)+\sqrt{(c^2-h^2)^2+4b^2c^2}\big),
		\end{align}
		\begin{align}\label{cs hat}
		c_s^2(\hat{\xi}):=\frac{1}{2}\big((c^2+h^2)-\sqrt{(c^2-h^2)^2+4b^2c^2}\big),
		\end{align}
		where $h^2=\lvert H\rvert^2/\rho$, $b^2=\lvert \hat{\xi}\times H\rvert^2/\rho.$
		
		Moreover, if we assume that $0<\lvert H\rvert^2\neq \rho c^2$, then we have\\
		$(i)$ When $\xi\cdot H\neq 0$ and $\xi\times H\neq 0$:\\
		$\lambda_0=\lambda_4$ is double eigenvalue of $A(U,\xi)$, and the eigenvalues $\lambda_{\pm 1},\lambda_{\pm 2}$ and $\lambda_{\pm 3}$ are simple eigenvalues  of $A(U,\xi)$.\\~~\\
		$(ii)$ When $\xi\cdot H=0$, $\xi\neq 0$:
		\\
		$\lambda_{\pm 3}$ are simple eigenvalues, while $\lambda_0=\lambda_{\pm 1}=\lambda_{\pm 2}=\lambda_4$ is a multiple eigenvalue.\\ ~~\\
		$(iii)$ When $\xi\times H=0$, $\xi\neq 0$: 
		\\ when $\lvert H\rvert^2<\rho c^2$ $(\text{resp.}\  \lvert H\rvert^2>\rho c^2)$, $\lambda_{\pm 3}$ $(\text{resp.}\  \lambda_{\pm 1})$ are simple; $\lambda_{+2}\neq \lambda_{-2}$ are double, equal to $\lambda_{\pm 1}$ $($resp. $\lambda_{\pm 3})$ depending on $\xi\cdot H$, $\lambda_0$ is double equal to $\lambda_4$.
	\end{lemma}
	The proof of this lemma is very similar to the explanation given in \cite[Appendix \ A]{metivierhypboundary} except here we have the additional eigenvalue $\lambda_4=0$. Also, here we will not state all the eigenspaces as in \cite{metivierhypboundary}. 
	\begin{proof}
		Let $\dot{U}=(\dot{\rho},\dot{u},\dot{H},\dot{p})$. The eigenvalue equation $A(U,\xi)\dot{U}=\lambda\dot{U}$ reads 
		\begin{align}
		\begin{cases}
		\lambda\dot{\rho}=\rho\dot{u}_{\parallel},\\
		\rho\lambda\dot{u}_{\parallel}=\dot{p}+H_{\perp}\cdot \dot{H}_{\perp},\\
		\lambda\rho\dot{u}_{\perp}=-H_{\parallel} \dot{H}_{\perp},\\
		\lambda\dot{H}_{\perp}=\dot{u}_{\parallel}H_{\perp}-H_{\parallel}\dot{u}_{\perp},
		\\\lambda\dot{H}_{\parallel}=0,
		\\\lambda\dot{p}=\gamma p \dot{u}_{\parallel}.
		\end{cases}
		\end{align}	
		On $\{\dot{\rho}=0,\dot{u}=0,\dot{H}_{\perp}=0,\dot{p}=0\}=\mathbb E_0(\xi)$, $A$ is equal to $\lambda=0$. From now on we work on $\mathbb E_0^{\perp}=\{\dot{H}_{\parallel}=0\}$ which is invariant by $A(p,\xi)$.
		
		Consider $v=H/\sqrt{\rho}$, $\dot{v}=\dot{H}/\sqrt{\rho}$, $\dot{\alpha}=\dot{p}/\rho$, $\alpha=p/\rho$, and $\dot{\sigma}=\dot{\rho}/\rho$. The characteristic system reads: 
		\begin{align}
		\begin{cases}
		\lambda\dot{\sigma}=\dot{u}_{\parallel},\\
		\lambda\dot{u}_{\parallel}=\dot{\alpha}+v_{\perp}\cdot \dot{v}_{\perp},\\
		\lambda\dot{u}_{\perp}=-v_{\parallel}\dot{v}_{\perp},
		\\\lambda\dot{v}_{\perp}=\dot{u}_{\parallel}v_{\perp}-v_{\parallel}\dot{u}_{\perp},\\
		\lambda\dot{\alpha}=\gamma \alpha \dot{u}_{\parallel}.
		\end{cases}
		\end{align}
		Take a basis of $\xi^{\perp}$ such that $v_{\perp}=(b,0)$ and let $a=v_{\parallel}$. In such a basis, the matrix of the system is
		\begin{align}\label{matrix}
		\lambda-\tilde{A}:=\begin{pmatrix}
		\lambda & -1& 0 & 0&0&0&0\\
		0& \lambda& 0&0&-b&0&-1\\
		0&0&\lambda&0&a&0&0\\
		0&0&0&\lambda&0&a&0\\
		0&-b&a&0&\lambda&0&0\\
		0&0&0&a&0&\lambda&0\\
		0&-\gamma \alpha &0&0&0&0&\lambda
		\end{pmatrix}.
		\end{align}
		The characteristic roots satisfy
		\begin{align}
		\lambda(\lambda^2-a^2)((\lambda^2-a^2)(\lambda^2-c^2)-\lambda^2b^2)=0
		\end{align}
		Thus either 
		\begin{align}
		&\lambda=0\ \text{or}\\ 
		&\lambda^2=a^2\ \text{or}\\
		&\lambda^2=c_f^2(\hat{\xi})=\frac{1}{2}(c^2+h^2+\sqrt{(c^2-h^2)^2+4 b^2c^2})\ \text{or}\\
		&\lambda^2=c_s^2(\hat{\xi})=\frac{1}{2}(c^2+h^2-\sqrt{(c^2-h^2)^2+4 b^2c^2}),
		\end{align}
		with $h^2=a^2+b^2=\lvert H\rvert^2/\rho$.
		
		As in Lemma $\ref{lemma char disjoint}$, if we consider $R(X)=(X-a^2)(X-c^2)-b^2X$, $\{R\leq 0\}=[c_s^2(\hat{\xi}),c_f^2(\hat{\xi})]$, and $R(X)\leq 0$ for $X\in[\min(a^2,c^2),\max(a^2,c^2)]$. Thus,
		\begin{align}
		c_f^2(\hat{\xi})\geq\max(a^2,c^2)\geq a^2,
		\end{align}
		\begin{align}
		c_s^2(\hat{\xi})\leq\min(a^2,c^2)\leq a^2.
		\end{align} 
		At the case $v_{\perp}\neq 0$ that is $w=\hat{\xi}\times v\neq 0$: we have the basis such that (\ref{matrix}) holds is smooth in $\xi$. In this basis, $w=(0,b)$, $b=\lvert v_{\perp}\rvert>0$.
		Since $R(c^2)=-b^2 c^2<0$ there holds $c_s^2(\hat{\xi})<c^2<c_f^2(\hat{\xi})$. Suppose that $a\neq 0$. Then $R(a^2)=-a^2 c^2<0$ and $c_s^2(\hat{\xi})<a^2<c_f^2(\hat{\xi})$. Moreover, $c_s^2(\hat{\xi})c_f^2(\hat{\xi})=a^2 c^2$ and $c_s^2(\hat{\xi})>0$. However, when $a=0$, we get $c_s^2(\hat{\xi})=0$, but $c_f^2(\hat{\xi})>c^2>0$. 
		
		When $a\neq 0$, and $b=0$, the eigenvalues of $\tilde{A}$ are $\pm c$ (simple), $0$ (simple), and $\pm h$ (double). Assume that $c^2\neq h^2$. Note that when $b=0$, then $\lvert a\rvert =h$ and 
		\begin{align*}
		\text{when}\ c^2>h^2: c_f(\hat{\xi})=c, \ c_s(\hat{\xi})=h,
		\end{align*}
		\begin{align*}
		\text{when}\ c^2<h^2: c_f(\hat{\xi})=h,\  c_s(\hat{\xi})=c.
		\end{align*}
	\end{proof}
	Let $Q$ be a pseudodifferential operator of order $1$, such that $Q\dot{U}=0$ be the system of the linearized ideal MHD equations, and $q=\sigma(Q)$ be its principal symbol.
	We have $\operatorname{det}q=\tau^2(\tau^2-c_s^2(\hat{\xi})\lvert\xi\rvert^2)(\tau^2-c_f^2(\hat{\xi})\lvert\xi\rvert^2)(\tau^2-(\xi\cdot H)^2/\rho)$. 
	\begin{prop}
		When we have $\Sigma$ is disjoint union of the hypersurfaces $S_1=\{q_1=\tau=0\}$, $S_2=\{q_2=\tau-c_s(\hat{\xi})\lvert\xi\rvert=0\}$, $S_3=\{q_3=\tau+c_s(\hat{\xi})\lvert\xi\rvert=0\}$, $S_4=\{q_4=\tau-(\xi\cdot H)/\sqrt{\rho}=0\}$, $S_5=\{q_5=\tau+(\xi\cdot H)/\sqrt{\rho}=0\}$, $S_6=\{q_6=\tau-c_f(\hat{\xi})\lvert\xi\rvert=0\}$, and  $S_7=\{q_7=\tau+c_f(\hat{\xi})\lvert\xi\rvert=0\}$; that is when we are outside the intersection of any of these hypersurfaces then $Q$ is of real principal type. Note that we have this case when $\xi\cdot H\neq 0$, and $\xi\times H\neq 0$.
	\end{prop}
	\begin{proof}
		Let $\Gamma_1$,....$\Gamma_7$ be the disjoint conic neighborhoods of $S_1$,...,$S_7$ respectively. 
		
		Let $^t q^{\operatorname{co}}$ be the adjugate matrix (transpose of the cofactor matrix) of $q$. We can check by using "Mathematica" for example that $^t q^{\operatorname{co}}$ can be written as 
		\begin{align}\label{adjugate app}
		^t q^{\operatorname{co}}=\tau M,
		\end{align}
		with $M$ being an $8\times 8$ matrix. 
		
		In $\Gamma_1$, we choose
		\begin{align}
		\tilde{q}=\bigg(1/\prod_{i=2}^{7}q_j\bigg)M,
		\end{align}
		so we get 
		\begin{align}
		\tilde{q}q=\tau \operatorname{Id}_8.
		\end{align}
		In $\Gamma_j$, for $j=2,...,7$ we choose
		\begin{align}
  \tilde{q}=\bigg(1/\bigg(q_1^2\prod_{\substack{i=2\\i\neq j}}^{7}q_i\bigg)\bigg) \ ^tq^{\operatorname{co}},
		\end{align}
		so we get 
		\begin{align}
		\tilde{q}q=q_j \operatorname{Id}_8.
		\end{align}
		As $q_j$ for $j=1,...7$ are of real principal type we get the result.
		
	\end{proof}
	
	Remember that from lemma \ref{12}, we know that $c_f^2(\hat{\xi})\neq c_s^2(\hat{\xi})$ and that $c_f^2(\hat{\xi})\neq 0$.
	\begin{prop}\label{prop uniaxial type application}
		If $\tau\neq 0$, and $\tau^2\neq c_f^2(\hat{\xi})\lvert\xi\rvert^2$, or if $\tau\neq 0$ and $\tau^2\neq c_s^2(\hat{\xi})\lvert\xi\rvert^2$, then our system is of uniaxial type at $\Sigma_2$. 
	\end{prop}
	\begin{proof}
		First case: If $\tau\neq 0$, and $\tau^2\neq c_f^2(\hat{\xi})\lvert\xi\rvert^2$, we have $\Sigma$ is union of two hypersurfaces $S_1=\{\tau-c_s(\hat{\xi})\lvert\xi\rvert=0\}\sqcup\{\tau+c_s(\hat{\xi})\lvert\xi\rvert=0\}$, and $S_2=\{\tau-(\xi\cdot H)/\sqrt{\rho}=0\}\sqcup\{\tau+(\xi\cdot H)/\sqrt{\rho}=0\}$ intersecting at $\Sigma_2=\{\tau=\lvert \xi\rvert\lvert H\rvert/\sqrt{\rho},\  \xi\times H=0, \lvert H\rvert^2<\rho c^2, \xi\neq 0\}\sqcup\{\tau=-\lvert \xi\rvert\lvert H\rvert/\sqrt{\rho},\  \xi\times H=0, \lvert H\rvert^2<\rho c^2, \xi\neq 0\}$.\\
		Second case: If $\tau\neq 0$, and $\tau^2\neq c_s^2(\hat{\xi})\lvert\xi\rvert^2$, we have $\Sigma$ is union of two hypersurfaces  $S_1=\{\tau-c_f(\hat{\xi})\lvert\xi\rvert=0\}\sqcup\{\tau+c_f(\hat{\xi})\lvert\xi\rvert=0\}$ and $S_2=\{\tau-(\xi\cdot H)/\sqrt{\rho}=0\}\sqcup\{\tau+(\xi\cdot H)/\sqrt{\rho}=0\}$ intersecting at $\Sigma_2=\{\tau=\lvert\xi\rvert\lvert H\rvert/\sqrt{\rho}, \xi\times H=0, \lvert H\rvert^2>\rho c^2, \xi\neq0\}\sqcup\{\tau=-\lvert\xi\rvert\lvert H\rvert/\sqrt{\rho}, \xi\times H=0, \lvert H\rvert^2>\rho c^2, \xi\neq0\}$.
		
		In the first and in the second case we have: $S_1$ and $S_2$ are tangent of order $1$ at $\Sigma_2$, 
		the codimension of $\Sigma_2$ is three, the (complex) dimension of $\mathcal N_Q$ is equal to $2$ at $\Sigma_2$, $d^2(\det q)\neq 0$ at $\Sigma_2$, and $d^i(\det q)=0$ at $\Sigma_2$ for $i<2$. Hence, the conditions $(\ref{2.1})$-$(\ref{2.4})$ are satisfied.
		It remains only to prove $(\ref{2.7})$. In \cite{denckerdoublerefraction}, Dencker mentioned that by proposition $3.2$ in \cite{denckerdoublerefraction}, we only have to verify
		\begin{align}
		\partial_\rho q:\ker q\mapsto \operatorname{Im} q\ \text{at}\ \Sigma_2
		\end{align}
		when $\rho\in T_{\Sigma_2}\Sigma$, since the order of tangency of $S_1$ and $S_2$ is $1$. 
		$T_{\Sigma_2}\Sigma$ is characterized as those $\rho\in T_{\Sigma_2}X$ such that $\partial_\rho^2(\operatorname{det}q)=0$. Thus $T_{\Sigma_2}\Sigma$ is spanned by $D_1=\xi_2\partial_{\xi_1}-\xi_1\partial_{\xi_2}$, $D_2=\xi_3\partial_{\xi_1}-\xi_1\partial_{\xi_3}$, $D_3=\xi_2\partial_{\xi_3}-\xi_3\partial_{\xi_2}$, $D_4=H_2\partial_{\xi_1}-H_1\partial_{\xi_2}$,
		$D_5=H_1\partial_{\xi_3}-H_3\partial_{\xi_1}$,
		$D_6=H_2\partial_{\xi_3}-H_3\partial_{\xi_2}$,
		$D_7=\partial_t$, $D_8=\xi_1\tau\partial_{\tau}+\lvert\xi\rvert^2\partial_{\xi_1}$, $D_9=\xi_2\tau\partial_{\tau}+\lvert\xi\rvert^2\partial_{\xi_2}$, $D_{10}=\xi_3\tau\partial_{\tau}+\lvert\xi\rvert^2\partial_{\xi_3}$, $D_{11}=\tau H_1\partial_{\tau}+(\xi\cdot H)\partial_{\xi_1}$, $D_{12}=\tau H_2\partial_{\tau}+(\xi\cdot H)\partial_{\xi_2}$, $D_{13}=\tau H_3\partial_{\tau}+(\xi\cdot H)\partial_{\xi_3}$. We can check that if $^t(\nu_1,...,\nu_8)\in\ker q$ at $\Sigma_2$, then we find 
		\begin{align}\label{Dj for the first and third case}
		D_j q^t(\nu_1,...,\nu_8)=\ ^t(0,...,0),\ \ \ \ j=1,...,13,
		\end{align}
		so, $(\ref{2.7})$ is satisfied. For $D_7$ we clearly have $D_7 q ^t\nu=0$. We will show how one can get $(\ref{Dj for the first and third case})$ for the other $D_j$, in particular we show the proof for $D_1$, and one can prove it for the other $D_j$ in a similar way. Let $\nu\in\ker q$ at $\Sigma_2$ so we have 
		\begin{align}\label{ker p2}
		\begin{cases}
		\pm \frac{\lvert H\rvert\lvert\xi\rvert}{\sqrt{\rho}} \nu_1+\rho \xi_1 \nu_2+\rho \xi_2 \nu_3+\rho \xi_3 \nu_4=0\\
		\pm \frac{\lvert H\rvert\lvert\xi\rvert}{\sqrt{\rho}} \nu_2-\rho^{-1}(H_3\xi_3+H_2\xi_2)\nu_5+\rho^{-1}\xi_1 H_2\nu_6+\rho^{-1}\xi_1H_3 \nu_7+\rho^{-1}\xi_1\nu_8=0\\
		\pm \frac{\lvert H\rvert\lvert\xi\rvert}{\sqrt{\rho}} \nu_3+\rho^{-1}\xi_2 H_1\nu_5-\rho^{-1}(H_3\xi_3+H_1\xi_1)\nu_6+\rho^{-1}\xi_2H_3 \nu_7+\rho^{-1}\xi_2\nu_8=0\\
		\pm \frac{\lvert H\rvert\lvert\xi\rvert}{\sqrt{\rho}} \nu_4+\rho^{-1}\xi_3 H_1\nu_5+\rho^{-1}\xi_3H_2 \nu_6-\rho^{-1}(H_2\xi_2+H_1\xi_1)\nu_7+\rho^{-1}\xi_3\nu_8=0\\
		-(H_2\xi_2+H_3\xi_3)\nu_2+\xi_2H_1\nu_3+\xi_3H_1\nu_4\pm\frac{\lvert H\rvert\lvert\xi\rvert}{\sqrt{\rho}}\nu_5=0\\
		\xi_1H_2\nu_2-(H_1\xi_1+H_3\xi_3)\nu_3+\xi_3H_2\nu_4\pm\frac{\lvert H\rvert\lvert\xi\rvert}{\sqrt{\rho}}\nu_6=0\\
		\xi_1H_3\nu_2+\xi_2H_3\nu_3-(H_1\xi_1+H_2\xi_2)\nu_4\pm\frac{\lvert H\rvert\lvert\xi\rvert}{\sqrt{\rho}}\nu_7=0\\
		\gamma p\xi_1\nu_2+	\gamma p\xi_2\nu_3+	\gamma p\xi_3\nu_4\pm \frac{\lvert H\rvert\lvert\xi\rvert}{\sqrt{\rho}}\nu_8=0.
		\end{cases}	\end{align}
		We have 
		\begin{align}
		\begin{split}
		D_1 q \nu &=D_1\begin{pmatrix}
		\tau \nu_1+\rho \xi_1 \nu_2+\rho \xi_2 \nu_3+\rho \xi_3 \nu_4\\ \tau \nu_2-\rho^{-1}(H_3\xi_3+H_2\xi_2)\nu_5+\rho^{-1}\xi_1 H_2\nu_6+\rho^{-1}\xi_1H_3 \nu_7+\rho^{-1}\xi_1\nu_8\\
		\tau \nu_3+\rho^{-1}\xi_2 H_1\nu_5-\rho^{-1}(H_3\xi_3+H_1\xi_1)\nu_6+\rho^{-1}\xi_2H_3 \nu_7+\rho^{-1}\xi_2\nu_8\\
		\tau \nu_4+\rho^{-1}\xi_3 H_1\nu_5+\rho^{-1}\xi_3H_2 \nu_6-\rho^{-1}(H_2\xi_2+H_1\xi_1)\nu_7+\rho^{-1}\xi_3\nu_8\\
		-(H_2\xi_2+H_3\xi_3)\nu_2+\xi_2H_1\nu_3+\xi_3H_1\nu_4+\tau\nu_5\\
		\xi_1H_2\nu_2-(H_1\xi_1+H_3\xi_3)\nu_3+\xi_3H_2\nu_4+\tau\nu_6\\
		\xi_1H_3\nu_2+\xi_2H_3\nu_3-(H_1\xi_1+H_2\xi_2)\nu_4+\tau\nu_7\\
		\gamma p\xi_1\nu_2+	\gamma p\xi_2\nu_3+	\gamma p\xi_3\nu_4+\tau\nu_8
		\end{pmatrix}\\
		&=D_1\begin{pmatrix}
		\tau \nu_1\mp \frac{\lvert H\rvert\lvert\xi\rvert}{\sqrt{\rho}} \nu_1\\ \tau \nu_2\mp \frac{\lvert H\rvert\lvert\xi\rvert}{\sqrt{\rho}} \nu_2\\
		\tau \nu_3\mp \frac{\lvert H\rvert\lvert\xi\rvert}{\sqrt{\rho}} \nu_3\\
		\tau \nu_4\mp \frac{\lvert H\rvert\lvert\xi\rvert}{\sqrt{\rho}} \nu_4\\
		\tau\nu_5\mp \frac{\lvert H\rvert\lvert\xi\rvert}{\sqrt{\rho}} \nu_5\\
		\tau\nu_6\mp \frac{\lvert H\rvert\lvert\xi\rvert}{\sqrt{\rho}} \nu_6\\
		\tau\nu_7\mp \frac{\lvert H\rvert\lvert\xi\rvert}{\sqrt{\rho}} \nu_7\\
		\tau\nu_8\mp \frac{\lvert H\rvert\lvert\xi\rvert}{\sqrt{\rho}} \nu_8
		\end{pmatrix}= \ ^t(0,...,0),
		\end{split}
		\end{align}
		for $\nu\in\ker q$ at $\Sigma_2$.	
	\end{proof}
	\begin{prop}
		If $\tau^2\neq c_f^2(\hat{\xi})\lvert\xi\rvert^2$ then  our system is of MHD type at $\Sigma_2$.
	\end{prop}
	\begin{proof}
		When $\tau^2\neq c_f^2(\hat{\xi})\lvert\xi\rvert^2$, then $\Sigma$ is union of the five hypersurfaces  $S_1=\{\tau=0\}$, $S_2=\{\tau-c_s(\hat{\xi})\lvert\xi\rvert=0\}$, $S_3=\{\tau+c_s(\hat{\xi})\lvert\xi\rvert=0\}$, $S_4=\{\tau-(\xi\cdot H)/\sqrt{\rho}=0\}$, and  $S_5=\{\tau+(\xi\cdot H)/\sqrt{\rho}=0\}$, intersecting at $\Sigma_2=\cap_{j=1}^5 S_j=\{\tau=0,\xi\cdot H=0, \xi\neq 0\}$. We want to prove that our system is of MHD type at $\Sigma_2$. We have $S_j$ intersect transversally at $\Sigma_2$, the codimension of $\Sigma_2$ is equal to two, $d^6(\det q)\neq 0$, and $d^i(\det q)=0$ for $i<6$ at $\Sigma_2$, and $\dim$ of the fiber of $\mathcal N_Q$ is equal to $6$ at $\Sigma_2$. $(\ref{detp at Si0})$ is satisfied for $i_0=1$.  Hence, still we want to check $(\ref{2.7 double refraction})$. Again, we will prove this by proving the following  
	\begin{align}
	\partial_\rho q: \ker q \mapsto \operatorname{Im} q\ \text{at}\ \Sigma_2,
	\end{align}
	when $\rho\in T_{\Sigma_2}\Sigma$.  $T_{\Sigma_2}\Sigma$ is spanned by $D_1=\xi_1\partial_{\xi_{1}}+\xi_2\partial_{\xi_{2}}+\xi_3\partial_{\xi_{3}}$, $D_2=\tau\partial_{\tau}$, $D_3=\partial_t$, $D_4=\tau\partial_{\xi_1}$, $D_5=\tau\partial_{\xi_2}$, $D_6=\tau\partial_{\xi_3}$, $D_7=\tau\partial_{x_1}$, $D_8=\tau\partial_{x_2}$, $D_9=\tau\partial_{x_3}$, $D_{10}=(\xi\cdot H)\partial_{x_1}$, $D_{11}=(\xi\cdot H)\partial_{x_2}$, $D_{12}=(\xi\cdot H)\partial_{x_3}$, $D_{13}=(\xi\cdot H)\partial_{\xi_1}$, and $D_{14}=(\xi\cdot H)\partial_{\xi_3}$ (note that we have not mentioned $(\xi\cdot H)\partial_{\xi_2}$ as it can be written in terms of $D_{13}$, $D_{14}$, and $D_{5}$). We have for $(\nu_1,...,\nu_8)\in\ker q$ at $\Sigma_2$
	\begin{align}
	D_i q\ ^t(\nu_1,...,\nu_8)= \ ^t(0,...,0) \ \text{at}\ \Sigma_2\ \text{for}\ i=1,...,14.
	\end{align}
	From $(\ref{adjugate app})$ we know that $(\ref{adjugate matrix})$, and $(\ref{principal symbol of L1})$ are satisfied with $R=D_t$, $\sigma(L_1)=M$, and $f=\tau(\tau^2-c_s^2(\hat{\xi})\lvert\xi\rvert^2)(\tau^2-c_f^2(\hat{\xi})\lvert\xi\rvert^2)(\tau^2-(\xi\cdot H)^2/\rho)$.
 \end{proof}
	\begin{remark}
		$\bullet$ When $\tau^2\neq c_f^2(\hat{\xi})\lvert\xi\rvert^2$, on $\Sigma\setminus \Sigma_2$ we have either $\tau\neq 0$ and for this case we have Proposition $\ref{prop uniaxial type application}$, or $\tau=0$ but $\xi\cdot H\neq 0$ so our system is of real principal type in this case.
		\\ $\bullet$ If $\tau^2\neq c_s^2(\hat{\xi})\lvert\xi\rvert^2$, but $\tau=0$ then we know that $\tau^2\neq c_f^2(\hat{\xi})\lvert\xi\rvert^2$ and again our system is of real principal type in this case.
	\end{remark}
	Note: As an application for systems of generalized transverse type, one can consider the linearized isentropic MHD equations, which is $7\times 7$ matrix; check \cite[Appendic A]{metivierhypboundary} where the first order term of the linearized isentropic MHD equations and its eigenvalues are given, and then we can easily check the type of the system as we did in this section for linearized ideal MHD equations.
 \nocite{*}
	\bibliography{reference}{}
	\bibliographystyle{amsplain}
	
\end{document}